\documentclass[12pt]{article}
\usepackage[T1]{fontenc}
\usepackage{multicol}
\usepackage{mathtools}
\usepackage{authblk}
\usepackage{dsfont}
\usepackage{hyperref}
\usepackage{bbm}
\usepackage{comment}
\hypersetup{
    colorlinks=true,
    pdftitle={Continuity of the normal cycle with respect to $C^0$ topologies.},
    pdfpagemode=FullScreen,
    }
\usepackage{amsmath,amsthm}
\usepackage{tikz}
\usepackage{amssymb}
\usepackage{changepage} 
\usetikzlibrary{cd}

\DeclareMathOperator{\vect}{\mathrm{Vect}}

\DeclareMathOperator{\clarke}{\partial^{*} \! \!}
\DeclareMathOperator{\Conv}{\mathrm{Conv}}

\DeclareMathOperator{\Nor}{\mathrm{Nor}}

\DeclareMathOperator{\dacy}{\delta_{\mathrm{Acy}}}
\DeclareMathOperator{\dhom}{d_{\mathrm{Hom}}}
\DeclareMathOperator{\dgm}{\mathrm{dgm}}

\DeclareMathOperator{\id}{\mathrm{id}}

\DeclareMathOperator{\one}{\mathbbm{1}}
\DeclareMathOperator{\diam}{\mathrm{diam}}
\DeclareMathOperator{\lip}{\mathrm{Lip}}
\DeclareMathOperator{\Supp}{\mathrm{supp}}
\DeclareMathOperator{\Vol}{\mathrm{Vol}}

\newcommand{\Sphere}{\mathbb{S}}
\DeclareMathOperator{\diff}{\mathrm{d} \!}
\DeclareMathOperator{\diffH}{\mathrm{d} \mathcal{H}}

\DeclareMathOperator{\tq}{ \, | \,  }

\newcommand{\R}{\mathbb{R}}
\newcommand{\N}{\mathbb{N}}
\newcommand{\module}[1]{\left\lvert #1 \right\rvert}
\newcommand{\slice}[1]{\langle [\diff h] \mres (U \times \R^d), \gamma, t \rangle}
\newcommand{\tranche}[3]{\langle #1, #2, #3 \rangle}
\newcommand{\norme}[1]{\left\lvert \left \lvert #1 \right \rvert \right \rvert}
\newcommand{\eps}{\varepsilon}

\newcommand{\scal}[2]{ \left \langle #1 , #2 \right \rangle}
\newcommand{\mres}{\mathbin{\vrule height 1.6ex depth 0pt width 0.13ex\vrule height 0.13ex depth 0pt width 1.3ex}}

\usepackage{cleveref}
\newtheorem{theorem}{Theorem}
\newtheorem*{theorem*}{Theorem}
\newtheorem{corollary}[theorem]{Corollary}
\newtheorem{lemma}[theorem]{Lemma} 
\newtheorem{definition}[theorem]{Definition}
\newtheorem{proposition}[theorem]{Proposition}
\newtheorem{remark}[theorem]{Remark}

\newtheorem*{unicite}{Uniqueness theorem for normal cycles}

\newtheorem*{MR}{Main Result}

\author[1]{David Cohen-Steiner}
\author[1]{Antoine Commaret}
\affil[1]{\textit{Centre INRIA d'Université Côte d'Azur, Valbonne, France}}

\title{Continuity of the normal cycle with respect to $C^0$-topologies}
\begin{document}

\maketitle

\begin{abstract}
We show that the generalized curvatures and more specifically the normal cycle of a compact subset of $\R^d$ are continuous under weak notions of convergence, assuming an a priori mass bound.
In particular, provided its mass remains bounded, the normal cycle behaves continuously when the subset is perturbed by a homeomorphism that is $C^0$-close to the identity. 
Our approach relies on a combination of persistent homology with the geometric measure theory framework classically used in the study of curvatures of singular sets.
As an application, we prove that every compact definable set in an o-minimal structure admits a normal cycle by showing that any such set is a limit, in the above sense, of a family of smooth sets whose normal cycles have uniformly bounded mass.
We also show that WDC sets are limits of nested smooth sets with uniformly bounded normal cycle masses. 
\end{abstract}

\section{Introduction}
Curvatures of possibly singular sets can be studied under various perspectives. In 1959, Federer \cite{CurvatureFederer} 
introduced a common framework for dealing
with curvatures of submanifolds of $\R^d$ and those of convex polyhedra by defining so-called \textit{curvatures measures} for sets with \textit{positive reach}. Curvature measures were subsequently extended to various classes of singular sets, such as polyconvex sets \cite{Polyconvex}, locally finite generic unions of sets with positive reach \cite{ZahleU_PR}, and sets definable in an o-minimal structure \cite{TameSets}. 
In parallel to these developments, a line of research emerged \cite{Wintgen, ZahleCurrent, ZahleUReach} that uses the framework of currents to construct an object containing richer information about the curvatures of singular sets. The \textit{normal cycle} is essentially the integral current associated with the unit normal bundle, and allows one to recover curvature measures and their tensorial variants \cite{BernigTensor, DavidTensor} by integrating certain universal differential forms against it.
Later, Fu \cite{FuSub} established a general, axiomatic characterization for normal cycles summarized as follows.

\begin{unicite}
        For any $X \subset \R^d$, there is at most one integral current $T$ which is a Legendrian cycle, and such that the Euler characteristic of $X$ intersected with almost every half-space of $\R^d$ can be recovered from $T$ using specific formulas - see \Cref{th:uniqueness} for the precise statement.
    If such a current exists, we call it the normal cycle of $X$ and denote it by $N_X$.
\end{unicite}

The uniqueness theorem may be seen as definition of the broadest possible class of subsets of $\R^d$ for which a normal cycle can be defined. The problem of finding alternate, more explicit characterizations of this class is not an easy one and remains a driving theme in the field.
Existence of normal cycles had been proved for submanifolds \cite{Wintgen}, sets of positive reach \cite{ZahleCurrent} and generic unions of those \cite{ZahleUReach}. Using the uniqueness theorem, it was established for several additional classes.
The original article \cite{FuSub} shows that any subanalytic set admits a normal cycle. Later \cite{RatajLipschitz} proved that Lipschitz submanifolds with locally bounded curvatures also admit a normal cycle, and more recently \cite{normalWDC} showed that WDC sets, i.e. sublevel sets of the difference of two convex functions at a weakly regular value, admit one as well. The latter result was extended to locally finite generic unions of WDC sets \cite{U_WDC}.

A method Fu suggests in \cite{FuPolyhedra} is the following. From the compactness theorem for currents of Federer-Fleming \cite[Section 4]{GMT}, the existence of the normal cycle of a set $X$ can be obtained by considering sequences of sets $X_n$ that are known to admit a normal cycle (e.g, sets of positive reach) and such that the $N_{X_n}$ have uniformly bounded mass. Indeed, extracting a subsequence from $X_n$ one can then assume that $N_{X_n}$ converges to some current $T$, which is necessarily a Legendrian cycle, so that it is only left to prove that $T$ recovers the Euler characteristic of $X$ intersected with almost any half-spaces of $\R^d$. Assuming we are able to do so, $X$ admits $T$ as its normal cycle. Moreover if this can be proven for any converging subsequence we have effectively proved that $N_{X_n}$ actually converges to $N_X$. In this paper, we exhibit weak topologies on compact subsets of $\R^d$ for which the Euler characteristic recovery property passes to the limit, thereby implying the continuity of the normal cycle. 

In what sense must a sequence of sets converge, in order to guarantee that
their curvature measures or normal cycles converge to that of the limit? The first to explicitly ask this question - for curvature measures - seems to be Milnor in his 1962 paper \cite{MilnorVolumes}. 
In \cite{FuPolyhedra}, Fu used the method outlined above to show that under some additional conditions on the so-called fatness of their simplices, convergence in the Hausdorff distance of a sequence of polyhedra to a submanifold with boundary of $\R^d$ implies convergence of the normal cycles to that of the submanifold.
Quantitative bounds on the speed of convergence of the normal cycles of a sequence of sets converging to a smooth submanifold are obtained in \cite{DavidTensor}. These bounds depend on the mass of the normal cycles of the approximating sets as well as the accuracy of their tangent spaces. 
A different but closely related line of research focuses on designing ways to approximate the curvature measures of a possibly nonsmooth set, from approximating sets whose curvature measures may not converge themselves. In \cite{BoundaryMeasures}, the authors approximate the curvature measures of a set with positive reach from a set lying at Hausdorff distance $\eps$, at a rate $O(\sqrt{\eps})$ in the Wasserstein distance.
The total mean curvature of a smooth domain is approximated in $O(\eps)$ in \cite{HerbertV1}.
More recently, it was shown in \cite{PersistentIntrinsicVolumes} that the total curvature measures (the \textit{intrinsic volumes}) of two sets that are $\eps$-close in the \textit{homotopy distance}, are close at a rate linear in $\eps$. The arguments used in this article are of integral geometric nature and do not seem well-suited for extending the result beyond intrinsic volumes, which are mere real numbers, to the whole normal cycle. We now give precise definitions to state our results.

\begin{definition}[Homotopy distance for subsets of $\R^d$]
The homotopy distance $\dhom(X,Y)$ between two closed subsets $X$ and $Y$ of $\R^d$ is the infimum of positive $\eps$ such that there exist $f : X \to Y$ and $g : Y \to X$ satisfying the following:
\begin{itemize}
    \item For all $x \in X$ (resp. $y \in Y$), $\norme{ f(x) - x } \leq \eps$ (resp. $\norme{ g(y) - y } \leq \eps)$, 
    \item There exist homotopies $H_X, H_Y$ respectively between $\id_X$ and $g \circ f$, and between $\id_Y$ and $f \circ g$, with trajectories moving points by less than $2\eps$. More precisely, there exist continuous maps ${H_X : [0,1] \times X \to X}$, ${H_Y : [0,1] \times Y \to Y}$, such that for all $x \in X$, $y \in Y$ and $t \in [0,1]$, we have:
    \[
    \left \{
    \begin{array}{c c l c c c l}
        H_X(0,x) &= &x & \quad & H_Y(0,y) & = & y \\
        H_X(1,x) &= &g(f(x)) & \quad & H_Y(1,y) & = & f(g(y)) \\
        \norme{H_X(t,x) - x} &\leq& 2 \eps & \quad & \norme{H_Y(t,y) - y} &\leq &2 \eps.\\ 
    \end{array}
    \right.\]
\end{itemize}
If no such $\eps$ exists (in particular if $X$ and $Y$ are not homotopy equivalent), we let $\dhom(X,Y) = + \infty$.
\end{definition}

We will also consider an analogous yet distinct notion of convergence for compact subsets of $\R^d$, which we call the \textit{acyclic convergence}.

\begin{comment}To that end we introduce some notations about relations between compact subsets $X, Y$ of $\R^d$, illustrated in the following diagram.
% https://q.uiver.app/#q=WzAsNCxbMSwxLCJBICJdLFswLDIsIlgnIl0sWzIsMiwiWSciXSxbMSwwLCJYIFxcdGltZXMgWSJdLFswLDMsIiIsMCx7InN0eWxlIjp7InRhaWwiOnsibmFtZSI6Imhvb2siLCJzaWRlIjoidG9wIn19fV0sWzAsMV0sWzAsMl0sWzMsMSwiXFxwaV9YIiwxXSxbMywyLCJcXHBpX1kiLDFdLFsxLDIsInBeQV9ZIiwyLHsib2Zmc2V0IjozLCJjdXJ2ZSI6LTQsInN0eWxlIjp7ImJvZHkiOnsibmFtZSI6ImRvdHRlZCJ9fX1dLFsyLDEsInBfWF5BIiwyLHsib2Zmc2V0IjotNCwiY3VydmUiOjQsInN0eWxlIjp7ImJvZHkiOnsibmFtZSI6ImRvdHRlZCJ9fX1dXQ==
\[
\begin{tikzcd}[row sep=2.2em, column sep=3em]
	& {A} & \\
	{X \supset X'} && {Y' \subset Y}
	\arrow["{\pi_X}|_{A}"{description},from=1-2, to=2-1]
	\arrow["{\pi_Y}|_{A}"{description},from=1-2, to=2-3]
	% Smooth looping dotted arrows
	\arrow["{p^A_Y}"', dotted, from=2-1, to=2-3, out=20, in=160]
	\arrow["{p_X^A}"', dotted, from=2-3, to=2-1, out=190, in=-10]
\end{tikzcd}
\]

Let $\pi_X$ (resp. $\pi_Y$) be the projection onto the first (resp. second) coordinate of $X \times Y$. For any relation $A  \subset X \times Y$, and any $X' \subset X, Y' \subset Y$, write $p^A_Y(X') = \pi_Y(\pi_X^{-1}(X') \cap A) $ and $p^A_X(Y') = \pi_X(\pi_Y^{-1}(Y') \cap A)$ for the subset of $Y$ (resp. $X$) consisting of point in relation with at least one element of $X'$ (resp. $Y'$).
When relation $A$ is clear from context, we may omit the overscript $A$ to ease notations.
\end{comment}

\begin{definition}[Acyclic convergence]
For a positive number $\eps$, a relation between $X, Y \subset \R^d$ is said to be an $\eps$-\textit{acyclic relation} if:
    \begin{itemize}
        \item The graph of the relation is closed;
        \item If $x$ and $y$ are in relation, then $\norme{x -y} \leq \eps$;
        \item For all $x \in X$ (resp. $y \in Y$), the set of points of $Y$ in relation with $x$ (resp. points of $X$ in relation with $y$) has the homology group of a point \footnote{Throughout the paper, we work with Čech homology with coefficients in an arbitrary field.}.
    \end{itemize}

We say that a sequence $X_n$ of compact subsets of $\R^d$ converges acyclically to $X \subset \R^d$ if there is a sequence of $\eps_n$-acyclic relations between $X_n$ and $X$ with $\eps_n$ converging to $0$.

\end{definition}

\begin{remark}[Special case of acyclic convergence]
    In particular, a sequence of compact subsets $X_n$ of $\R^d$ converges acyclically to $X$ when there exist continuous, surjective maps $f_n : X_n \to X$ such that
    \begin{itemize}
        \item For every $x$ in $X$, $f_n^{-1}(\{x \})$ has the homology group of a point - in this case, $f_n$ is said to \emph{acyclic};
        \item The sequence $\eps_n \coloneqq \sup \left \{ \norme{f_n(z) - z}, z \in X_n \right \}$ converges to 0.
    \end{itemize}
\end{remark}

Note that in a sequence converging for the homotopy distance, the subsets must eventually share the same homotopy type as the limit, whereas this is not necessarily the case for the acyclic convergence.\\

Our main result is that normal cycles are continuous with respect to both the homotopy distance and the acyclic convergence, assuming an a priori bound on the mass of normal cycles, or equivalently, on the $(d-1)$-volume of the unit normal bundles.

\begin{MR}
\label{MR}
   Let $X_n$ be a sequence of compact subsets of $\R^d$ with smooth boundaries either converging to $X$ acyclically or with respect to the homotopy distance.
   If the normal cycles of $X_n$ have uniformly bounded mass, then $X$ admits a normal cycle $N_X$, and the sequence of normal cycles $N_{X_n}$ converges to $N_X$.
\end{MR}

The above result implies in particular the convergence of curvature measures, which may come as a surprise since the homotopy distance or the acyclic convergence do not even provide any control on the convergence of tangent spaces.
The theorem also provides a new class of sets admitting a normal cycle -- namely sets $X$ that are limits of subsets as specified in the statement. The extent of this class remains to be explored, nonetheless it is easy to see that it contains sets with positive reach, and the arguments used in \cite{RatajLipschitz} also show that it contains the class of Lipschitz submanifolds with locally bounded curvatures. We further show that compact sets definable in an o-minimal structure also belong to this class, thereby providing the first complete proof of the existence of normal cycles for such sets. The proof of \cite{Bernig}, which follows a different approach, contains a gap, as observed by \cite{Nicolaescu}. We finally show that WDC sets belong to this class as well. Our proof allows us to express them as a decreasing intersection of smooth domains with normal cycles of uniformly bounded mass, thereby answering a question raised at the end of \cite{KinematicWDC}.

\section{Currents and normal cycles}

 The space of $k$-currents is the space dual to smooth $k$-differential forms with compact supports on $\R^d$, the same way distributions are defined with respect to test functions. In the following paragraphs, we provide some notations for currents and state the theorems and properties that are relevant to this article. For more details about currents, we refer the reader to the classical text \cite[Section 4]{GMT}, or \cite{GMTintro} for an easier introduction.\\

\begin{definition}[Operations on currents]
\label{def:operation_currents}
Let $T$ be a $k$-current in a Euclidean space $\R^m$.
\begin{itemize}
    \item Its boundary is the $(k-1)$-current $\partial T : \omega \mapsto T(\diff \omega)$.
    \item Let $\phi$ be a $p \leq k$ differential form. Then the restriction of $T$ to $\phi$ is the current $(k-p)$-current defined by
    \begin{equation*}
        T \mres \phi : \omega \mapsto T(\phi \wedge \omega).
    \end{equation*} 
    \item For any smooth function $h : \R^m \to \R^{m'}$,
    the pushforward of $T$ is a $k$-current $h_* T$ in $\R^{m'}$ defined by
        \begin{equation*}
            h_* T(\psi) \coloneqq T(h^* \psi),
        \end{equation*}
    where $h^* \psi$ is the pullback of $\psi$ by the map $h$.
\end{itemize}
\end{definition}

Throughout the paper, we will use two different norms on the space of currents.
\begin{definition}[Norms on currents {\cite[Chapter 4]{GMT}}]
The mass of a current $T$ is defined as 
\begin{equation*}
    M(T) \coloneqq \sup \{ T(\phi), \norme{\phi}_{\infty} \leq 1 \}. 
\end{equation*}
The flat norm $F$ of $T$ is defined as follows:
\begin{equation*}
    F(T) \coloneqq \inf \{ M(A) + M(B), T = A + \partial B \}.
\end{equation*}
\end{definition}

Convergence in the flat norm is more relaxed than convergence in mass. In fact, it is equivalent to the weak convergence, implying the lower semicontinuity of the mass with respect to the flat norm.

 \begin{proposition}[Lower semicontinuity of mass with respect to the flat convergence]
     Let $T_j$ be a sequence of currents with supports included in some compact $K$ and converging to a current $T$ with respect to the flat norm. Then 
     \begin{equation*}
         M(T) \leq \liminf_{i \to + \infty} M(T_i).
     \end{equation*}
 \end{proposition}

The currents appearing in this paper (i.e., normal cycles, Monge-Ampère currents, or any current stemming from such currents and the operations defined in \Cref{def:operation_currents}, as well as the later defined slices) are all \textit{integral currents}. 
The exact definition of integral currents is not needed for the present paper, and we refer the reader to \cite[Section 4]{GMT} for a precise description. For any integral current $T$ and any set $U \subset \R^m$, the map $\phi \mapsto T( \one_U \phi)$ is consistently defined - even though $\one_U \phi$ might not be smooth.
In spite of the slight abuse in notation, this map is traditionally denoted $T \mres U$, and its mass is $M(T \mres U) = \sup \{ T(\phi),  \; \Supp(\phi) \subset U, \norme{\phi}_{\infty} \leq 1  \}$. \\

Another important property of integral currents is the existence of a \textit{slicing} operation. Prosaically, the slice $\tranche{T}{\pi}{v}$ of a current $T$ by a Lipschitz map $\pi$ can be thought as the current $T$ restricted to the fiber $\pi^{-1}(v)$. 

\begin{definition}[Slicing of a current]
 Let $\pi : \R^m \to \R^{m'}$ be a locally Lipschitz function and let $T$ be a $k$-integral current on $\R^m$ with $k \geq m'$. Then the slices of $T$ by $\pi$ are $(k - m')$ integral currents $\tranche{T}{\pi}{v}$ defined for almost any $v \in \R^d$, satisfying, for every $(k-m')$ differential form $\psi$
    \begin{equation}
    \label{eq:slice}
        \int_{v \in \R^{m'}} \phi(\nu) \tranche{T}{\pi}{v}(\psi(v)) = T \mres \pi^* (\phi \wedge \Omega)(\psi),
    \end{equation} 
    where $\Omega$ is the volume form of $\R^{m'}$.
\end{definition}

The following instance of the compactness theorem for integral currents is key to our approach.

\begin{theorem}[Particular case of the compactness theorem for currents with bounded mass {\cite[4.2.17]{GMT}}]
\label{th:currents_compactness}
Let $T_i$ be a sequence of $k$-integral currents such that $\partial T_i = 0$ for all $i \in \N$ and such that $\liminf_{i \to \infty} M(T_i) < + \infty$. Then there exists a $k$-integral current $T$ and a subsequence of $T_i$ converging to $T$ in the flat norm.
\end{theorem}

The uniqueness theorem of Fu provides a definition of normal cycles in the broadest setting.
There exist several ways of formulating it. We use here the same presentation as \cite[Chapter 9]{RatajZahle}.

\begin{theorem}[Uniqueness theorem for normal cycles {\cite[Section 3]{FuSub}}]
\label{th:uniqueness}
Let $X$ be a closed subset of $\R^d$. Then there exists at most one $(d-1)$-integral current $T$ with support in $\R^d \times \Sphere^{d-1}$ satisfying the following three properties:
\begin{itemize}
    \item $T$ is a cycle, that is, $\partial T = 0$,
    \item $T$ is Legendrian, that is, $T \mres \alpha =0$, where $\alpha = \sum_{i=1}^d x_i \diff y_i$ is the canonical contact form,
    \item For almost all $(\nu,t)$ in $\Sphere^{d-1} \times \R$, we have
    \[ p_{\#} \langle T, \pi_1, - \nu \rangle (\mathbf{1}_{(-\infty, t]}) = \chi( X \cap H_{\nu,t}),\]
    where $p : (x, \nu) \mapsto \scal{x}{\nu}$, $\pi_1 : (x,n) \mapsto n$ and ${H_{\nu, t} \coloneqq \{ x \in \R^d, \scal{x}{\nu} \leq t \} }$.
\end{itemize}
    If such a current exists, we call it the normal cycle of $X$ and denote it by $N_X$.
\end{theorem}

\section{Persistence theory and Čech homology}
\label{sec:persistance}
This section gives a summary of the classical notions of persistence theory to make the paper more self-contained. Further information about this topic can be found in \cite{StructureStability}.

\begin{definition}[Persistence modules]
Let $\mathbb{K}$ be a field. A persistence module $M$ is a functor $\R \to \vect_{\mathbb{K}}$, that is, a collection of $\mathbb{K}$-vector spaces $(M_t)_{t \in \R}$ and linear maps $\phi^{t}_s : M_s \to M_t$ for any $s \leq t \in \R$, such that $\phi^u_{t} \circ \phi^{t}_s = \phi^{u}_s$ for any ordered triple $s \leq t \leq u$ in $\R$. It is said to be \textit{$q$-tame} when for every $s < t$ in $\R$ the maps $\phi^{t}_s$ have finite rank.
\end{definition}

\begin{comment}
    From the functorial definition of persistence modules,
\end{comment} 
There is a natural notion of morphisms between persistence modules, from which one can define a (pseudo)metric on persistence modules called the \textit{interleaving distance}.

\begin{definition}[Morphisms, $\delta$-interleavings and interleaving distance]
Let $M,N$ be two persistence modules and let $\delta \geq 0$.\\
\begin{itemize}
    \item $M^{\delta}$ denotes the persistence module $(M_{t + \delta})_{t \in \R}$, i.e., $M$ shifted by $\delta$.  
    \item A morphism $j$ between two persistence modules $N$ and $M$ is a natural transformation between functors, i.e., a collection of linear maps $(j_t)_{t \in \R} : N_t \to M_t$ such that the left diagram in \Cref{fig:commute} commutes.
    \item  The modules $N$ and $M$ are $\delta$-\emph{interleaved} when there exist two morphisms  $u, v$ respectively from $M$ to $N^{\delta}$ and from $N$ to $M^{\delta}$ such that the right diagram in \Cref{fig:commute} commutes, 

    \item The \textit{interleaving distance} between $M$ and $N$ is defined as:
\[d_I(M,N) \coloneqq \inf \{ \delta \in \R^+ \tq M \text{ and } N \text{ are } \delta \text{-interleaved}\}. \]

    \item If $M, N$ are graded persistence modules with same degrees, the interleaving distance between them is the maximum of the interleaving distance between their respective persistence modules of each degree.
\end{itemize}
\end{definition}

% https://q.uiver.app/#q=WzAsMTIsWzAsMSwiTl9zIl0sWzAsMCwiTV9zIl0sWzEsMSwiTl90Il0sWzEsMCwiTl90Il0sWzQsMF0sWzQsMV0sWzUsMCwiTV97cytcXGRlbHRhfSJdLFs1LDEsIk5fe3MrXFxkZWx0YX0iXSxbMywwLCJNX3MiXSxbMywxLCJOX3MiXSxbNywxLCJOX3tzKzJcXGRlbHRhfSJdLFs3LDAsIk1fe3MrMlxcZGVsdGF9Il0sWzAsMSwial9zIl0sWzIsMywial90IiwyXSxbMSwzXSxbMCwyXSxbOCw2XSxbOSw3XSxbNywxMF0sWzYsMTFdLFs5LDYsInZfcyIsMSx7ImxhYmVsX3Bvc2l0aW9uIjoyMCwiY3VydmUiOjF9XSxbOCw3LCJ1X3MiLDEseyJsYWJlbF9wb3NpdGlvbiI6MjAsImN1cnZlIjotMX1dLFs3LDExLCJ2X3tzK1xcZGVsdGF9IiwxLHsibGFiZWxfcG9zaXRpb24iOjMwLCJjdXJ2ZSI6MX1dLFs2LDEwLCJ1X3tzK1xcZGVsdGF9IiwxLHsibGFiZWxfcG9zaXRpb24iOjMwLCJjdXJ2ZSI6LTF9XV0=
\begin{figure}[h!]
    \centering
\[\begin{tikzcd}
	{M_s} & {M_t} && {M_s} & {} & {M_{s+\delta}} && {M_{s+2\delta}} \\
	{N_s} & {N_t} && {N_s} & {} & {N_{s+\delta}} && {N_{s+2\delta}}
	\arrow["{j_s}", from=2-1, to=1-1]
	\arrow["{j_t}"', from=2-2, to=1-2]
	\arrow[from=1-1, to=1-2]
	\arrow[from=2-1, to=2-2]
	\arrow[from=1-4, to=1-6]
	\arrow[from=2-4, to=2-6]
	\arrow[from=2-6, to=2-8]
	\arrow[from=1-6, to=1-8]
	\arrow["{v_s}"{description, pos=0.25}, from=2-4, to=1-6]
	\arrow["{u_s}"{description, pos=0.25}, from=1-4, to=2-6]
	\arrow["{v_{s+\delta}}"{description, pos=0.25}, from=2-6, to=1-8]
	\arrow["{u_{s+\delta}}"{description, pos=0.25}, from=1-6, to=2-8]
\end{tikzcd}\]
    \caption{Commutative diagrams in the definition of morphisms (left) and $\delta$-interleaving (right).}
    \label{fig:commute}
\end{figure}
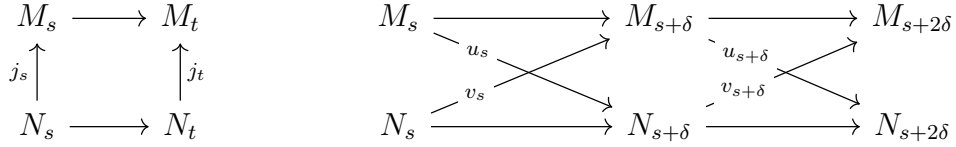

Under suitable mild conditions, persistence modules can be decomposed as the direct sum of so-called \textit{interval modules}. From this decomposition, we can associate to a persistence module its \textit{persistence diagram}. 
\begin{comment}
    This decomposition can be described using \textit{persistence diagrams}.
\end{comment}

\begin{definition}[Persistence diagrams of interval decomposable persistence modules]
Given an interval $I \subset \R$, $\one_{I}$ is the persistence module defined by
\[ 
(\one_{I})_t \coloneqq \left \{
\begin{array}{cc}
     \mathbb{K} & \text{ when } t \in I  \\
     0 & \text{ otherwise} 
\end{array} \right. \]
and such that the linear maps between $(\one_{I})_s \to (\one_{I})_{t} $ are the identity when $s < t \in I$.
A persistence module $M$ is said to be \emph{interval decomposable} when it is isomorphic to a module of the form
\[
\bigoplus_{I \in \mathcal{I}} \one_{I},
\]
where $\mathcal{I}$ is an at most countable multiset of intervals of $\R$. \\

The \emph{persistence diagram} associated to an interval decomposable persistence module is the multiset of points in $\R^2$ whose coordinates are the bounds of the intervals decomposing $M$:
\begin{equation*}
    \dgm(M) \coloneqq \bigsqcup_{I \in \mathcal{I}} \{ (\inf I, \sup I) \}.
\end{equation*}

Elements of a persistence diagram are usually referred to as either \emph{intervals} or \emph{bars}.
We denote by $N(M)$ the number of intervals of an interval decomposable persistence module $M$.
\end{definition}

The interleaving distance between persistence modules, of algebraic nature, has its  counterpart at the diagram level, defined in a combinatorial way.

\begin{definition}[Bottleneck distance]
 A $\delta$-matching between two persistence diagrams $D, D'$ is a bijective map  $\gamma : C \to C'$, between subsets of $D, D'$ such that for any $c \in C$, $\norme{\gamma(c) - c}_{\infty} \leq \delta$, and such that for any $(a,b) \in (D \setminus C) \cup (D' \setminus C')$, $\module{a-b} \leq 2\delta$.

 The \emph{bottleneck distance} between two diagrams $D, D'$ is defined as:
\[ d_B(D, D') \coloneqq \inf \{ \delta \tq \text{There exists a } \delta \text{-matching between } D \text{ and } D' \}.\]

\end{definition}

The isometry theorem \cite[Chapter 5]{StructureStability} of persistence theory states that these two distances coincide: for any pair of interval decomposable persistence modules $M,N$, we have $d_I(M,N) = d_B(\dgm(M), \dgm(N))$.\\

We will study persistence modules arising from sublevel set filtrations.

\begin{definition}[Persistent homology notations for sublevel set filtrations]
 Let $f : \R^d \to \R$ be a continuous function and let $X$ be a closed subset of $\R^d$. For $t \in \R$ let $X^f_t = X \cap f^{-1}(-\infty, t]$ be the sublevel set filtration of $X$. Applying the $i$-th homology functor to this filtration yields a persistence module which we denote by $M_i(f, X)$.
 
 The graded persistence module with $M_i(f,X)$ at degree $i$ is the \emph{persistent homology module} associated to $f|_{X}$, and we denote it by $M(f,X)$.
 
 When the modules $M_i(f,X)$ are interval decomposable, we denote their associated persistence diagrams by $\dgm_i(X,f)$. The \emph{persistent homology diagram} $\dgm(X,f)$ is the graded persistence diagram with $\dgm_i(X,f)$ as its component of degree $i$.  
\end{definition}

In the remainder of the paper, when dealing with persistent homology of sublevel set filtration, we will always work with Čech homology. 
Defined as the inverse limit of the homology of nerves of coverings - see \cite{Spanier}, it coincides with singular homology on sufficiently regular sets, but enjoys the following continuity property.

\begin{proposition}[Right continuity of Čech homology for sublevel set filtrations of continuous functions]
\label{prop:cech_continuity}
Let $f : \R^d \to \R$ be a continuous function and $X \subset \R^d$. 
Considering Čech homology, for any $s$ in $\R$ we have the following inverse limit:
\begin{equation*}
\begin{comment}
\label{eq:cech_continuity}
\end{comment}
\lim_{t \to s^+} H_i(X^f_t) = H_i(X^f_s).
\end{equation*}
In particular, for any $s \in \R$ and nonzero $u \in H_i(X_s)$, there exists $t > s$ such that the image of $u$ by the linear map induced by the inclusion $H_i(X^f_s) \to H_i(X^{f}_t)$ is nonzero. 
\end{proposition}

Any persistence module satisfying the last conclusion of \Cref{prop:cech_continuity}, and that can furthermore be approximated up to arbitrary precision by modules with uniformly bounded numbers of intervals, is finitely decomposable into intervals.

\begin{proposition}[Semi-continuity of the number of bars]
\label{prop:semi_bars}
    Let $M$ be a persistence module such that for any $s \in \R$, $\lim_{t \to s^+} M_t = M_s$.
    Then $M$ is interval decomposable.
    Furthermore, assume that there is a sequence $M(n)$ of interval decomposable persistence modules such that $\lim_{n \to + \infty} d_I(M, M(n)) = 0$. Then
    \begin{equation*}
        N(M) \leq \liminf_{n \to + \infty} N(M(n)).
    \end{equation*}
\end{proposition}

\begin{proof}
    First, remark that $M$ is $q$-tame. Indeed, let $\eps_n = d_I(M, M(n))$ and consider $s < t$ in $\R$. For $n$ large enough,  $s + 2\eps_n < t$ so that by definition of the interleaving distance the map $M_s \to M_t$ factors through $(M(n))_{s + \eps_n}$, so its rank is at most $N(M(n))$, which is finite for large enough $n$. Following \cite[Section 3.2]{observable}, a $q$-tame persistence module satisfying the last conclusion of \Cref{prop:cech_continuity} is equal to its radical, implying that it is interval decomposable and that its intervals cannot be singletons.
    Extracting a subsequence, assume that $N(M(n)) = C \coloneqq \liminf_{k \to + \infty} N(M(k))$ for all $n \in \N$. By the isometry theorem, there exists an $\eps_n$-matching between $\dgm(M(n))$ and $\dgm(M)$, implying that the number of intervals of $M$ with length greater than $2\eps_n$ is at most $C$. Letting $\eps_n$ go to zero yields the desired result since the intervals of $M$ all have positive length.
\end{proof}

It is not hard to see that the two above propositions imply the following.

\begin{corollary}[Shape of intervals arising in Čech homology]
\label{cor:shape_intervals}
Any persistence module arising from Čech homology is interval decomposable and has its intervals left-closed and right-open. 
\end{corollary} 

If $D$ is a graded, finite persistence diagram, define its Euler characteristic map as the alternating sum (in degrees) of the numbers of pairs $(a,b)$ such that $a \leq t < b$. Equivalently, seeing pairs $(a,b)$ as  left-closed and right-open intervals:
\[\chi(D) : t \mapsto \sum_{\substack{I \in D \\ \text{ of degree } i}} (-1)^{i} \one_{I}(t),\]
so that when $D$ is the diagram of a module coming from a sublevel set filtration, it just gives the Euler characteristic of the sublevel sets. 

The following result allows us to bound the average difference between Euler characteristic of graded persistence diagrams a function of their bottleneck distance and the number of intervals in them.
The gist of the proof is that given an $\eps$-matching, the symmetric difference of two matched intervals has length at most $2\eps$, while unmatched intervals have length at most $2\eps$.

\begin{lemma}
[Special case of the $\chi$-averaging lemma {\cite[Lemma 5.4]{PersistentIntrinsicVolumes}}]
\label{lem:averaging}
Let $D, D'$ be two finite graded persistence diagrams with $d_B(D, D') \leq \eps$. For a persistence diagram $D$, let $N(D)$ be the total number of intervals in it. We have:
\begin{equation*}
 \int_{\R} \module{\chi(D(t)) - \chi(D'(t))} \mathrm{d}t \leq 2 \eps (N(D) + N(D')).
\end{equation*}
\end{lemma}

\section{Continuity of normal cycles}

This section is dedicated to proving our main result, stated at the end of the introduction. The argument works by showing that from any subsequence of $N_{X_n}$, we can further extract a subsequence converging to a current, which has to be $N_X$ by Fu's uniqueness result. We fix the following notations for the remainder of the article. 
\begin{itemize}
    \item Let $\dacy(X,Y)$ be the infimum of the $\eps$ such that there exists an $\eps$-acyclic relation between $X$ and $Y$.
    
    \item Let $\eps_n \coloneqq \dhom(X,X_n)$ if $X_n$ is assumed to converge with respect to the homotopy distance, or $\eps_n = \dacy(X_n,X)$ if $X_n$ converges acyclically. By either assumption the sequence $\eps_n$ is nonnegative and converges to zero.  
    \item For any $\nu \in \Sphere^{d-1}$ let $h_{\nu} : x \mapsto \scal{\nu}{x}$, and for any real $t$ let $H_{\nu,t}$ be the closed half-space $h_{\nu}^{-1}(-\infty,t ]$. Note that there is a one-to-one correspondence between $\Sphere^{d-1} \times \R$ and closed half-spaces of $\R^d$ via $(\nu, t) \mapsto H_{\nu, t}$.
    \item For any $(\nu, t) \in \Sphere^{d-1} \times \R$ and for any $Z \subset \R^d$ put $Z^{\nu}_{t} = Z \cap H_{\nu,t}$.

    \item For any $Z \subset \R^d$ denote by $M_i(\nu, Z)$ the persistence module 
    \begin{equation*}
        (H_i(Z^{\nu}_t))_{t \in \R} = (H_i(Z \cap H_{\nu, t}))_{t \in \R},
    \end{equation*}
    that is, the $i$-th persistent Čech homology module associated to the sublevel filtration of $h_{\nu}$ restricted to $Z$. We denote by $\dgm_i(\nu,Z)$ its persistence diagram.
    
    \item We denote by $N(\nu,Z)$ the total number of topological events in $\dgm(\nu,Z)$, that is, the sum over $i$ of numbers of finite bounds of intervals in $\dgm_i(\nu,Z)$. Note that when $(h_{\nu})|_{Z}$ is Morse, this quantity coincides with the number of critical points of $(h_{\nu})|_{Z}$.

\end{itemize}

We will make use of the Vietoris-Begle theorem, which provides isomorphisms in  Čech homology from an acyclic map, that is, a continuous map $f:X \to Y$ whose fibers $f^{-1}(y)$ have the homology group of a point for all $y \in Y$ - note in particular that any acyclic map must be surjective.

\begin{proposition}[Vietoris-Begle \cite{vietoris-begle}]
\label{prop:vietoris_begle}
Let $f: Y \to X$ be an acyclic map between two compact metric spaces. Then for Čech homology over any coefficient ring, in any degree $i$, the linear map induced by $f$, $f_* : H_i(Y) \to H_i(X)$ is an isomorphism.
\end{proposition}

Key to our approach is the fact that our notions of convergence imply the convergence of persistence diagrams of height functions.

\begin{lemma}[Comparison of persistence modules of height functions]
\label{lem:convergence_lemma}
Let $X$ and $Y$ be two compact subsets of $\R^d$ such that either {$\dacy(X,Y)$} or {$\dhom(X,Y)$} is at most $\eps > 0$.
Then for every $\nu$ we have: 
\[ d_I(M(\nu, X), M(\nu, Y)) \leq \eps. \]
\end{lemma}

\begin{proof}
If $\dhom(X,Y) \leq \eps$, we have $d_I(M(\nu, X), M(\nu, Y)) \leq \eps$ as in \cite[Theorem 1.1]{PersistentIntrinsicVolumes}.

Else let $\eps' > \eps \geq \dacy(X,Y)$, let $A \subset X \times Y$ be an $\eps'$-acyclic relation between $X$ and $Y$.
Denote by $\pi_X : A \to X$ and $\pi_Y : A \to Y$ the projections on the first and second factor of $X \times Y$. For $Z_1 \subset X, Z_2 \subset Y$ let $p_X(Z_2) = \pi_X(\pi_Y^{-1}(Z_2))$ and $p_Y(Z_1) = \pi_Y(\pi_X^{-1}(Z_1))$ be the subset of $X$ (resp. $Y$) consisting of elements in relation with at least one element in $Z_2$ (resp. $Z_1$).
Since each $h_{\nu}$ is 1-Lipschitz, we have $Y^{\nu}_t \subset p_Y(X^{\nu}_{t + \eps'}) \subset Y^{\nu}_{t + 2\eps'}$.
Moreover, by the Vietoris-Begle theorem, both maps $\pi_X, \pi_Y$ restricted to $\pi_X^{-1}(X^\nu_{t+\varepsilon})$ induce isomorphisms in each homology degree.  This leads to the following commutative diagram where horizontal maps are induced by inclusions, solid vertical maps by $\pi_X$ or $\pi_Y$, dotted vertical maps by their inverse.

% https://q.uiver.app/#q=WzAsMTAsWzEsMCwiSF9pKFlee1xcbnV9X3QpIl0sWzEsMV0sWzIsMCwiSF9pKHBfWShYXlxcbnVfe3QrXFx2YXJlcHNpbG9ufSkpKSJdLFszLDAsIkhfaShZXntcXG51fV97dCsyXFx2YXJlcHNpbG9ufSkiXSxbMiwxLCJIX2koXFxwaV9YXnstMX0oWF5cXG51X3t0K1xcdmFyZXBzaWxvbn0pKSJdLFswLDBdLFswLDFdLFs0LDFdLFs0LDBdLFsyLDIsIkhfaShYXntcXG51fV97dCtcXHZhcmVwc2lsb259KSJdLFswLDJdLFsyLDNdLFs1LDAsIiIsMCx7InN0eWxlIjp7ImJvZHkiOnsibmFtZSI6ImRhc2hlZCJ9fX1dLFszLDgsIiIsMCx7InN0eWxlIjp7ImJvZHkiOnsibmFtZSI6ImRhc2hlZCJ9fX1dLFs0LDldLFs5LDQsIiIsMSx7ImxhYmVsX3Bvc2l0aW9uIjo2MCwiY3VydmUiOi0xLCJzdHlsZSI6eyJib2R5Ijp7Im5hbWUiOiJkb3R0ZWQifX19XSxbMCw5LCIiLDEseyJvZmZzZXQiOjJ9XSxbOSwzLCIiLDEseyJvZmZzZXQiOjJ9XSxbNCwyXSxbMiw0LCIiLDAseyJvZmZzZXQiOi0xLCJjdXJ2ZSI6LTEsInN0eWxlIjp7ImJvZHkiOnsibmFtZSI6ImRvdHRlZCJ9fX1dXQ==
\[\begin{tikzcd}
	{} & {H_i(Y^{\nu}_t)} & {H_i(p_Y(X^\nu_{t+\varepsilon'}))} & {H_i(Y^{\nu}_{t+2\varepsilon'})} & {} \\
	{} & {} & {H_i(\pi_X^{-1}(X^\nu_{t+\varepsilon'}))} && {} \\
	&& {H_i(X^{\nu}_{t+\varepsilon'})}
	\arrow[dashed, from=1-1, to=1-2]
	\arrow[from=1-2, to=1-3]
	\arrow[shift right=2, from=1-2, to=3-3]
	\arrow[from=1-3, to=1-4]
	\arrow[shift left, bend left = 30, dotted, from=1-3, to=2-3]
	\arrow[dashed, from=1-4, to=1-5]
	\arrow[from=2-3, to=1-3]
	\arrow[from=2-3, to=3-3]
	\arrow[shift right=2, from=3-3, to=1-4]
	\arrow[bend left = 30, shift left, dotted, from=3-3, to=2-3]
\end{tikzcd}\]

Chasing the diagram we end up with linear maps $H_i(Y^{\nu}_t) \to H_i(X^{\nu}_{t+\eps'})$ and $H_i(X^{\nu}_{t+\eps'}) \to H_i(Y^{\nu}_{t+2\eps'})$, which after composition yield the inclusion-induced map $H_i(Y^{\nu}_t) \to H_i(Y^{\nu}_{t + 2\eps'})$. Considering as well the linear maps obtained similarly by exchanging the role of $X$ and $Y$, we have effectively built an $\eps'$-interleaving between modules $M_i(\nu, X)$ and $M_i(\nu, Y)$. Letting $\eps'$ go to $\eps$ yields the desired result.
\end{proof}

Next, we show that the average number of critical points of height functions on a smooth manifold is bounded by the mass of its normal cycle.

\begin{lemma}[Bound on the average number of critical points of height functions]
\label{lem:crit_mass}
Let $Z$ be a smooth compact domain of $\R^d$. Then
\begin{equation*}
\int_{\nu \in \Sphere^{d-1}} N(\nu, Z) \diff \mathcal{H}^{d-1}(\nu) \leq M(N_{Z}).
\end{equation*}   
\end{lemma}

\begin{proof}
    It is a classical result that given a smooth domain $Z$, for almost-all $\nu$ in $\Sphere^{d-1}$, $(h_{\nu})|_{Z}$ is a Morse function, the critical points of this function being the points $x$ at which $-\nu$ is an outward pointing normal. Denoting by $\Nor(Z) \subset Z \times \Sphere^{d-1}$ the unit normal bundle of $Z$, this amounts to $N(\nu, Z) = \# \left \{ x \in Z, (x,-\nu) \in \Nor(Z) \right \}$. Now by a change of variables using the projection map $\pi : \Nor(Z) \to \Sphere^{d-1}$, $(x, p) \mapsto p$, the coarea formula \cite[3.2.22]{GMT} reads
    \begin{align*}
        \int_{\Sphere^{d-1}} N(\nu, Z) \diffH^{d-1}(\nu) = & \int_{\Sphere^{d-1}} \# \pi^{-1}( \{ \nu \}) \diffH^{d-1}(\nu) \\
        = & \int_{\Nor(Z)} J_{\pi}(x,p) \diffH^{d-1}(x,p),
    \end{align*}
    where $J_{\pi}(x,p)$ is the Jacobian of $\pi$ at $(x,p)$. As $\pi$ is a projection this Jacobian is at most one, yielding that the integral is at most  $\int_{\Nor(Z)} 1 \diff \mathcal{H}^{d-1}(x) = M(N_{Z})$.
\end{proof}

\begin{lemma}
\label{lem:lim_inf&tame}
Let $X$ be a compact subset of $\R^d$ and assume $X$ is the limit of a sequence of smooth domains $X_n$ with respect to either the acyclic convergence or the homotopy distance. Assume further that $\liminf_{n\to \infty} M(N_{X_n}) < + \infty$, where $N_{X_n}$ denotes the normal cycle of $X_n$. Then we have

    \begin{equation*}
    \int_{\Sphere^{d-1}} \liminf_{n \to \infty} N(\nu, X_n) \diff \nu \leq \liminf_{n \to \infty} M(X_n) 
    \end{equation*} 
Moreover, for almost every $\nu$ in $\Sphere^{d-1}$,
    \begin{equation*}
    N(\nu, X) \leq \liminf_{n \to \infty} N(\nu, X_n) < + \infty.
    \end{equation*}
\end{lemma}

\begin{proof}

 The first inequality follows from the mass bound of \Cref{lem:crit_mass} and Fatou's lemma. This implies that the quantity $\liminf_{n \to \infty} N(\nu, X_n)$ is finite for almost every $\nu \in \Sphere^{d-1}$. To get the bound on $N(\nu, X)$, note that thanks to  \Cref{lem:convergence_lemma}, acyclic or homotopy distance convergence implies that the modules $M_i(\nu, X_n)$ converge to $M_i(\nu, X)$ in the interleaving distance. From \Cref{prop:semi_bars}, the bound on the number $N(\nu, X)$ of intervals of $M(\nu,X)$ follows.
 
\end{proof}

Under the same assumptions, we show that the map $\psi_n : (\nu, t) \mapsto \chi(X_n \cap H_{\nu,t})$ converges to $\psi : (\nu, t) \mapsto \chi(X \cap H_{\nu,t})$.

\begin{lemma}[$L^1$ convergence of $\psi_n$]
\label{lem:L1_conv}
Let $X$ be a compact subset of $\R^d$ and assume $X$ is the limit of a sequence of smooth domains $X_n$ with respect to either the acyclic convergence or the homotopy distance. Assume further that there exists a positive $M$ such $M(N_{X_n}) \leq M$ for all $n$. Then 
\begin{equation*}
    \int_{\Sphere^{d-1} \times \R} \module{\psi_n(\nu, t) - \psi(\nu,t)} \diffH^{d}(\nu,t) = \norme{\psi_n - \psi}_{1, \Sphere^{d-1} \times \R} \xrightarrow[n \to \infty]{} 0. 
\end{equation*}
\end{lemma}
\begin{proof}
By \Cref{lem:convergence_lemma} and the isometry theorem between persistent modules and their associated persistent diagrams, we have for every $\nu$
\begin{equation*}
    d_B(\dgm(\nu, X),\dgm(\nu, X_n)) = d_I(M(\nu,X), M(\nu, X_n)) \leq \eps_n.
\end{equation*} 

Applying the $\chi$-averaging Lemma \ref{lem:averaging} then yields
\begin{equation}
\begin{aligned}
\label{eq:borne_R}
\int_{\R} \module{\psi_n(\nu,t) - \psi(\nu,t)} \diff t = & \int_{\R} \module{ \chi(X \cap H_{\nu,t}) - \chi(X_n \cap H_{\nu,t})} \diff t \\
 \leq & \; 2 \eps_n (N(\nu, X_n) + N(\nu, X)). 
\end{aligned}
\end{equation}    

Since the $X_n$ are smooth domains, each map $\psi_n$ is measurable.
Assuming $\psi$ is measurable, using the upper bounds of \Cref{lem:lim_inf&tame}, we obtain
\begin{equation*}
\begin{aligned}\label{eq:psi_n}
            \int_{\Sphere^{d-1} \times \R} \module{\psi_n(\nu, t) - \psi(\nu,t)} \diffH^{d}(\nu,t) \leq 2 \eps_n (M(X_n) + M).
\end{aligned}
\end{equation*}
from which the desired result follows.

However for now we only know that for all $\nu$, the map $t \mapsto \psi(\nu,t)$ is measurable, as these maps are linear combinations of indicator functions of the intervals in $M(\nu,X)$. The inequality in \Cref{eq:borne_R} shows that for almost all $\nu$, $ t \mapsto \psi_n(\nu, t) - \psi(\nu, t)$ converges to 0 in $L^1(\R)$. 

To show that $\psi$ is measurable, observe that applying the triangle inequality to \Cref{eq:borne_R}, integrating over the sphere and using the upper bounds of \Cref{lem:lim_inf&tame} we get
\[ \norme{\psi_i - \psi_j}_{1, \Sphere^{d-1} \times \R} \leq 2\max(\eps_{i}, \eps_j)(M(X_i) + M(X_j) + 2 M). \]
In particular the maps $\psi^*_n \coloneqq \psi_n - \psi_1$ are $L^1$ in $\Sphere^{d-1} \times \R$ and form a Cauchy sequence, which must converge to some map $\Psi^*$ in $L^1(\Sphere^{d-1} \times \R)$. Extracting a subsequence of $\psi^*_n$, we can assume that $\psi^*_n$ also converges almost everywhere to $\Psi^*$. Now $\psi^*_n(\nu, \cdot)$ converges to both $\Psi^*(\nu, \cdot)$ and $\psi(\nu, \cdot) - \psi_1(\nu, \cdot)$, so that almost everywhere in $\Sphere^{d-1} \times \R$ we have $\psi = \Psi^* + \psi_1$, hence $\psi$ is measurable. 

  \end{proof}

We are now in position to prove our main theorem.
\begin{theorem}[Continuity of normal cycles]
\label{th:MR}
   Let $X_n$ be a sequence of smooth compact subsets of $\R^d$, converging to $X \subset \R^d$ either acyclically or with respect to homotopy distance, and denote its associated sequence of normal cycles by $N_{X_n}$. 
    Assume that the sequence $M(N_{X_n})$ is uniformly bounded, the sequence of normal cycles $N_{X_n}$ converges in flat-norm to $N_X$.
\end{theorem}

\begin{proof}
    We follow the classical method first developed by Fu in \cite{FuSub}. Thanks to the compactness theorem on currents, we can extract a subsequence $X_{\phi(n)}$ such that $N_{X_{\phi(n)}}$ converges to a limit $T$. Recall that normal cycles $N_{X_n}$ are Legendrian cycles characterized by the fact that, for almost all $(\nu, t) \in \Sphere^{d-1} \times \R$, 
    \begin{equation*}
        p_{\#} \langle N_{X_n}, \pi_1, - \nu \rangle (\mathbf{1}_{(-\infty, t]}) = \chi(X_n \cap H_{\nu,t}).
    \end{equation*}
    Since $\psi_{\phi(n)} - \psi$ converges to $0$ in $L^1$ by \Cref{lem:L1_conv}, we can extract a subsequence $\psi_{\phi(\phi'(n))}$ converging almost everywhere to $\psi$, that is,
    \begin{equation*}
    \label{eq:euler_limite}
        \lim_{n \to \infty} \chi(X_{\phi(\phi'(n))} \cap H_{\nu,t}) =  \chi( X \cap H_{\nu, t}) \; \mathrm{a.e.}
    \end{equation*}
    Now since $N_{X_{\phi(n)}}$ converges in flat norm to $T$, for almost all $(\nu, t)$ in $\Sphere^{d-1} \times \R$ by a classical fact of measure theory (see for instance Prop 1.66 in \cite{RatajZahle}) we have
    \begin{equation*}
    \label{eq:courant_limite}
        \lim_{n \to \infty} p_{\#} \langle N_{X_{\phi(n)}}, \pi_1, - \nu \rangle (\mathbf{1}_{(-\infty, t]}) = p_{\#} \langle T, \pi_1, - \nu \rangle (\mathbf{1}_{(-\infty, t]}).
    \end{equation*}

    Combining the three previous equations we have on a set of full measure in $\Sphere^{d-1} \times \R$
    \begin{equation*}
        p_{\#} \langle T, \pi_1, - \nu \rangle (\mathbf{1}_{(-\infty, t]}) = \chi( X \cap H_{\nu,t}).
    \end{equation*}
    $T$ being a Legendrian current without boundary, Fu's uniqueness result implies that $T$ is indeed the normal cycle of $X$.
    Moreover, the arguments in this proof apply for every subsequence of $X_n$, 
    implying that the only limit point of the sequence $N_{X_n}$ is $N_X$. The result follows from the compactness theorem.
\end{proof}

\begin{remark}[Extension to nonsmooth sets]
    \label{rm:domains}
    \Cref{th:MR} holds if we only assume that $X_n$ has a positive reach or is a Lipschitz domain whose complement has positive reach. Indeed, in these cases almost all height functions $h_{\nu}|_{X_n}$ are Morse, by \cite{MorseReach} for sets of positive reach and by \cite{GeneralizedMorseTheory} for the second case, so the line of reasoning starting in \Cref{lem:crit_mass} still holds.
\end{remark}

\section{Approximable inverse flows of finite length and definable sets} 

In this section, we apply our main theorem to prove that any compact set $X$ definable in an o-minimal structure admits a normal cycle which is the limit of the normal cycles of its offsets $X^{\eps} \coloneqq \{ x \in \R^d, d_X(x) \leq \eps \}$ where $d_X$ denotes the distance function to $X$. More generally, we prove that the 0-sublevel set of a locally Lipschitz function $f$ admits a normal cycle to which the normal cycles of $\eps$-sublevel sets converge in the flat norm when $\eps$ goes to zero, under certain conditions on $f$ that are satisfied by distance functions.\\

Our idea is the following. When $f$ is $C^1$, one can consider the normalized inverse flow of $f$ consisting of trajectories of the differential equation $x'(t) = -\nabla f(x(t))/\norme{\nabla f(x(t))}$ (provided the gradient does not vanish). Any such trajectory is $1$-Lipschitz and makes $f$ decrease at speed $\norme{\nabla f(x(t))}$; it is either  a path of infinite length or it reaches a critical point of $f$. Just as in the traditional proof in Morse theory, assuming $a < b$ in $\R$ are such that $(a,b]$ does not contain critical values of $f$, if trajectories are of uniformly bounded length, the normalized inverse flow yields a homotopy-preserving map between sublevel sets $f^{-1}(-\infty, b]$ and $f^{-1}(-\infty, a]$. We build on this idea to construct homotopy equivalences with trajectories of bounded lengths between $f^{-1}(-\infty, \eps]$ and $f^{-1}(-\infty, 0]$, under assumptions on the \textit{Clarke gradient} of $f$, a classical object in Lipschitz analysis generalizing the gradients of $C^1$ functions.

\begin{definition}[Clarke gradient]
Let $f : \R^d \to \R$ be a locally Lipschitz map. The \textit{Clarke gradient} of $f$ at $x \in \R^d$, denoted by $\clarke f(x)$, is the convex hull of gradients of $f$ near $x$:
\begin{equation*}
\clarke f(x) = \Conv \left ( \{ \lim_{i} \nabla f(x_i) \tq  x_i \to x \} \right ).
\end{equation*} 
\end{definition}

The critical points of a locally Lipschitz map $f : \R^d \to \R$ are the points $x$ such that $0 \in \clarke f(x)$. To measure how far a point is to be critical, we define the distance to 0 of the Clarke gradient as the norm of its element closest to zero:
\begin{equation*}
    \Delta ( \clarke f(x) ) \coloneqq \inf \{ \norme{u}, u \in \clarke f(x) \}. 
\end{equation*}

A \textit{weakly regular value} of $f$ is a real $c$ such that there exist $\eps, \mu > 0$ such that for every point $x \in f^{-1}(c, c+\eps]$, we have $\Delta( \clarke f(x)) \geq \mu$.\\

In the one dimensional case, flows of finite length are easily characterized. 

\begin{definition}[Inverse flow of finite length]
\label{def:finite_length}
    Let $\phi : (0,K) \to (0, + \infty)$ be a locally Lipschitz map  for some possibly infinite positive $K$.
    For $C$ in $(0,K)$, let $T(C)$ be the length of the maximal solution satisfying $f(0) = C$ of the differential equation
        \begin{equation*}
        f'(t) = -\phi(f(t)),        
    \end{equation*}
    We say that $\phi$ has
    inverse flow of finite length when there exists $C$ such that $T(C)$ is finite.
\end{definition}

    The above definition is equivalent to asking that $1/\phi$ is integrable at 0. Indeed, the inverse $g = f^{-1} : (0, C] \to [0, T(C))$ of $f$ has derivative $g'(t) = -1/\phi(t)$, so that
    \begin{equation*}
        \int_{0}^C \frac{1}{\phi(t)} \diff t = T(C).
    \end{equation*}

Following ideas from \cite{Kurdyka} developed in the setting of $C^1$ definable functions, 
we will show that the following condition on the distance to 0 of Clarke gradients, already appearing in \cite{nonsmooth-loja}, suffices in building retraction with trajectories of bounded length.

\begin{definition}[]
\label{def:approx_finite_length}
    We say that a locally Lipschitz function $g : \R^d \to \R$ satisfies a nonsmooth Kurdyka-Łojasiewicz inequality when there exists $\phi$ with inverse flow of finite length and a positive number $c$ such that, on $g^{-1}(0, c]$, we have
    \begin{equation*}
         \Delta(\clarke g(x)) \geq \phi(g(x)).
    \end{equation*}

\end{definition}

	If $\Delta(\clarke g(x)) \geq \mu$ for some positive constant $\mu$ on a set of the form $g^{-1}(0, \eps]$, then clearly $g$ satisfies such an inequality, implying that every Lipschitz map with 0 as a weakly regular value also does. It was proved in \cite[Theorem 11]{nonsmooth-loja} that any locally Lipschitz function definable in any o-minimal structure satisfies such an inequality. The classical Łojasiewicz inequality is a particular case of Kurdyka-Łojasiewicz inequality where $\phi(t) = C t^{\theta}$ for some positive $C$ and $\theta$ in $(0,1)$. However some o-minimal structures contain definable maps that do not satisfy any classical Łojasiewicz inequalities.\\

    \begin{lemma}[Retractions of bounded length]
    \label{lem:retractions}
    Let $g$ be a function satisfying a nonsmooth Kurdyka-Łojasiewicz inequality, and let $\phi, c$ be as in \Cref{def:approx_finite_length}. Let $K = \int_{0}^c \frac{\diff s}{\phi(s)}$ and $K_{\eps} = \frac{1}{1-\eps} K$ for every $0 <\eps < 1$. Then there exists a continuous map $C_{\eps} : [0,K_{\eps}] \times g^{-1}(-\infty,c] \to g^{-1}(-\infty, c]$ such that :
    \begin{itemize}
        \item $C_{\eps}$ is $C^{\infty}$ on $\left \{ (t,x), g(C_{\eps}(t,x)) > 0 \right \}$, and on this set 
        \begin{equation*}
            \frac{\partial (g \circ C_{\eps}) }{\partial t}(t,x) \leq -(1 - \eps) \phi(g(C_{\eps}(t,x)));
        \end{equation*}
        \item $C_{\eps}(0,x) = x$ for every $x$ in $g^{-1}(-\infty,c]$;
        \item $C_{\eps}(t,x) = x$ for every $t \in [0,K_{\eps}]$ and $x \in g^{-1}(-\infty, 0]$;
        \item $C_{\eps}(K_{\eps},x) \in g^{-1}(-\infty, 0]$ for every $x$ in $g^{-1}(-\infty,c]$;
        \item For every $x$ in $g^{-1}(-\infty,c]$, the trajectory $[0,K_{\eps}] \ni t \mapsto C_{\eps}(t,x)$ is 1-Lipschitz.
    \end{itemize}
    \end{lemma}

    \begin{proof}
    We proceed as in the proof of \cite[Proposition 2.9]{GeneralizedMorseTheory}, adapting for the case where $\liminf_{s \to 0} \phi(s) = 0$. Let $\eps > 0$. For every $x \in g^{-1}(0, c]$, by convexity of $\clarke g(x)$, its element of smallest norm $W(x)$ satisfies  
    \begin{equation}
    \label{eq:hyper_support_clarke}
    \forall \, u \in \clarke g(x), \scal{u}{W(x)} \geq \norme{W(x)}^2 = \Delta( \clarke g(x))^2.    
    \end{equation}

    Since $\phi(g(x))$ is positive, by the upper semicontinuity of the Clarke gradient, there exists an open set $U_x \subset g^{-1}(0, + \infty)$ such that $\forall y \in U_x, \clarke g(y)$ is a subset of the $\phi(g(x)) \eps/2$-offset of $\clarke g(x)$.
    Since $y \mapsto \phi(g(y))$ is also continuous, there exists an open set $V_x \subset g^{-1}(0, + \infty)$ such that  
    \[ \forall y \in V_x, \; \phi(g(y)) \leq \frac{1 - \eps/2}{1 - \eps} \phi(g(x)). \]
    Letting $B_x = U_x \cap V_x$, we have for every $y \in B_x$ and every $u \in \clarke g(y)$
    \begin{equation}
        \label{eq:inegalite_clarke}
        \begin{aligned}
        \scal{u}{\frac{W(x)}{\norme{W(x)}}} \geq & \; \Delta( \clarke g(x)) - \frac{\eps}{2} \phi(g(x)) 
        \geq & \; (1 -\eps/2) \phi(g(x)) \\
        \geq & \; (1 - \eps)\phi(g(y)).
        \end{aligned}
    \end{equation}
    The set $B = \cup_{x \in g^{-1}(0,c]} B_x$ is open and contains $g^{-1}(0,c]$. As an open set of Euclidean space it is a paracompact, so that we can extract a countable, locally finite family $(B_{x_i})_{i \in I}$ with an associated locally finite $C^{\infty}$ partition of unity $(\rho_i)_{i \in I}$. Let $V$  be a smooth interpolation of the normalized $-W(x_i)$:
    \begin{equation*}
        V(x) = - \sum_{i \in I} \rho_i(x) \frac{W(x_i)}{\norme{W(x_i)}}.
    \end{equation*}
    In particular, for any $y \in g^{-1}(0, c]$, the nonzero contributions in $V(y)$ come from points $x_i$ such that $y \in B_{x_i}$. From \Cref{eq:inegalite_clarke} it then follows that for any $u \in \clarke g(y)$, we have 
    \begin{equation}
    \begin{aligned}
    \label{eq:inegalite_V}
    \scal{u}{V(y)} \leq -(1 - \eps) \phi(g(y)).
    \end{aligned}
    \end{equation}
        
    Let $C_V$ be the flow of maximal domain satisfying $\frac{\partial}{\partial t} C_{V}(t,x) = V(C_{V}(t,x))$.
    Since by construction $\norme{V}_{\infty} \leq 1$, for any $x_0$ in $g^{-1}(0, c]$ the trajectory $x(t) : t \mapsto C_{V}(t,x)$ is 1-Lipschitz. Moreover this trajectory is defined in an interval of the form $[0, T(x_0))$ where $T(x_0)$ belongs to $(0, + \infty ]$. Now let $\psi(t) = g(x(t))$, so that $\psi$ is a locally Lipschitz function $[0, T(x_0)) \to \R$. As such, it is differentiable almost everywhere in this interval. Take $s$ a point of differentiability of $\psi$. By definition, we have $\psi(s + h) = \psi(s) + h\psi'(s) + o(h)$, and by construction
    \begin{equation*}
        \begin{aligned}
            \psi(s + h) = & \, g(x(s+h)) \\
            = & \, g(x(s) + h V(x(s))) + o(h). 
        \end{aligned}
    \end{equation*}
    As a consequence, $g$ admits a directional derivative at $x(s)$ in direction $V(x(s))$ that is equal to $\psi'(s)$. By the characterization of the Clarke gradient found in (Proposition 1.4, \cite{Clarke1975GeneralizedGA}), we have
    \begin{equation*}
    \psi'(s) \leq \max \{ \scal{u}{V(x(s))}, u \in \clarke g(x(s)) \}.
    \end{equation*}
    By \Cref{eq:inegalite_V}, we get for almost every $t \in (0, T(x_0))$
    \begin{equation}
    \label{eq:inegalite_diff}
        \frac{ \partial g(x(t))}{\partial t} \leq -(1 - \eps) \phi(g(x(t)). 
    \end{equation}
    Note that this implies that $g(x(t))$ is decreasing in $t$, so that the restriction of $C_{\eps}$ to points of $g^{-1}(0, c]$ has trajectories in $g^{-1}(0,c]$, and so that 
    \begin{equation}
    \label{eq:Tx_0}
    T(x_0) = \sup \{ t \in \R_{\geq 0}, g(x(t)) > 0 \}.
    \end{equation} 
    
    Since $\phi$ is locally Lipschitz, we can consider the solution $h$ of the first order differential equation $h'(t) = (1-\eps)\phi(h(t))$ with starting point $g(x_0)$, with maximal domain $[0, T_h(x_0))$. Since $(1 -\eps)\phi$ has inverse flow of finite length, $T_h(x_0) = (1 - \eps)^{-1} \int_{0}^{g(x_0)} \phi(t)^{-1} \diff t$. Comparing the differential inequality of \Cref{eq:inegalite_diff} with the differential equation satisfied by $h$ yields $g(x(t)) \leq h(t)$ for every $0 < t < \min( T(x_0), T_h(x_0))$. By \Cref{eq:Tx_0} and the similar characterization for $T_h(x_0)$, we have 
    \[ T(x_0) \leq  T_h(x_0) = \frac{1}{(1 - \eps)} \int_{0}^{g(x_0)} \frac{1}{\phi(t)} \diff t \leq K_{\eps}. \] 
    
    Furthermore, as the curve $[0, T(x_0)) \ni t \mapsto x(t)$ is 1-Lipschitz by completeness it has an endpoint $C_V(T(x_0),x_0)$ which verifies $g(x(T(x_0))) = 0$. 
    Let $C_{\eps}$ be defined on $[0, K_{\eps}] \times g^{-1}(-\infty, c]$ by
    \begin{equation*}
        C_{\eps}(t,x) = \left \{
    \begin{array}{l l}
         C_{V}(\min(t, T(x)), x) & \text{if } x \in g^{-1}(0,c]  \\
         x & \text{else.} 
    \end{array}
        \right.
    \end{equation*}
    The only remaining property to be shown for $C_{\eps}$ is its continuity, which can be derived exactly as in the end of the proof of \cite[Proposition 2.9]{GeneralizedMorseTheory}. 
    \end{proof}
    
    \begin{theorem}[$C^0$ convergence of sublevel sets]
    \label{th:approx_acy_convergence}
        Let $g$ be a function satisfying a Kurdyka-Łojasiewicz  inequality, and let $c > 0$, $\phi$ be as in \Cref{def:approx_finite_length}. 
        Then letting $X = g^{-1}(-\infty, 0]$ and $X_t = g^{-1}(-\infty,t]$, we have for $0 < t \leq c$:
        \begin{equation*}
        \label{acy_bound}
        \max \left (\dacy(X,X_t), \dhom(X,X_t) \right) \leq \int_{0}^t \frac{\diff s }{\phi(s)}.
        \end{equation*}
    \end{theorem}

    \begin{proof}
        For any $0 < t \leq c$, $g$ satisfies the Kurdyka-Łojasiewicz  inequality for the pair $\phi, t$, so that for every $0 <\eps < 1$
        \Cref{lem:retractions} provides a map $C^t_{\eps} : [0, K_{\eps}(t)] \times X_t \to X_t$ where $K_{\eps}(t) = (1 - \eps)^{-1} \int_{0}^t \phi(s)^{-1} \diff s.$ The map $C^t_{\eps}$ is a deformation retraction between $X$ and $X_t$, and since its trajectories $s \mapsto C^t_{\eps}$ are 1-Lipschitz, it is a $K_{\eps}(t)$-homotopy equivalence. Thus we have
        \begin{equation*}
            \dhom(X, X_t) \leq \frac{1}{1-\eps} \int_{0}^t \frac{1}{\phi(s)} \diff s.
        \end{equation*}
        
        Moreover, let $F : X_t \to X : x \mapsto C_{\eps}(K_{\eps}(t),x)$. This map is acyclic, as for each $x \in X$ the fiber $F^{-1}(x)$ can be retracted to $x$ by the map $C^{t}_{\eps}$. Thus the relation on $X \times X_t$ where $x \in X$ is in relation with $y \in X_t$ when $F(y) = x$ is acyclic. Moreover, the trajectories of $C^t_{\eps}$ are 1-Lipschitz, so that each $y \in F^{-1}(x)$ is at distance at most $K_{\eps}(t)$ to $x$, so that this relation is a $K_{\eps}(t)$-acyclic relation between $X$ and $X_t$. Thus we also have
        \begin{equation*}
        \dacy(X, X_t) \leq \frac{1}{1-\eps} \int_{0}^t \frac{1}{\phi(s)} \diff s.
        \end{equation*}
        Letting $\eps$ go to zero in the two previous inequalities yields the desired bounds.
 \end{proof}

Recall that a function $g$ is said to be semiconcave when there exists $K \geq 0$ such that $x \mapsto g(x) - K \norme{x}^2$ is concave, and that $g$ is semiconvex when $-g$ is semiconcave. From Bangert \cite{Bangert}, if $t$ is a regular value of $g$, $X_t$ is a domain with positive reach (if $g$ is semiconvex) or a Lipschitz domain whose complement set has positive reach (if $g$ is semiconcave). In the first case the normal cycle of $X_t$ is given by the Lipschitz manifold of pairs $(x,v)$ where $v$ is a unit vector in the normal cone of $X_t$ at $x$, and in the second case it is obtained from the complement's normal cycle by flipping the normals, see \cite[Example 9.10]{RatajZahle}.  By \Cref{rm:domains} our main result \Cref{th:MR} yields the following.

    \begin{proposition}[Normal cycles of sublevel sets]
    \label{th:normal_cycle_flow}
        Let $g$ be a semiconcave function or semiconvex function satisfying a Kurdyka-Łojasiewicz  inequality and let $c, \phi$ be as in \Cref{def:approx_finite_length}.
        For $0 < \eps < c$, let $N_{\eps}$ be the normal cycle of $X_{\eps}$. If
    \begin{equation*}
        \limsup_{\eps \to 0} M(N_{\eps}) < \infty,
    \end{equation*}
    then $X$ admits a normal cycle $N_X$ and $N_{\eps}$ converges to $N_X$ in the flat norm as $\eps$ goes to zero.
    \end{proposition}

    \begin{corollary}
    Any closed subset $X$ of $\R^d$ definable in an o-minimal structure admits a normal cycle. 
    \end{corollary}
    \begin{proof}
    The squared distance $d^2_X$ is semiconcave and we have $X = (d^{2}_X)^{-1}(-\infty, 0]$. Moreover since $X$ is definable in an o-minimal structure, $d^2_X$ is definable and must satisfy a Kurdyka-Łojasiewicz inequality by \cite{nonsmooth-loja}. To prove the result, it remains to show that the normal cycles of small offsets of $X$ have uniformly bounded mass. Let $t$ be such that $(0,t]$ does not contain a critical value of $d_X$.
    By \cite[Theorem 2.20]{GeneralizedMorseTheory}, the normal cones of the complement of $X^s$ are opposite to the cones spanned by the Clarke gradients of $d_X$, so that the definable set 
    \[ A = \left \{ (x, v/\norme{v}), 0 < d_X(x) \leq t, v \in \clarke d_X(x) \right \}, \]
    intersected with $d_X^{-1}(s) \times \R^d$ is the support $A_s$ of  $N_{X^s}$ for $0 < s \leq t$; moreover we have $M(N_{X_s}) = \mathcal{H}^{d-1}(A_s)$. Now let $G$ be the Grassmannian of affine $(d+1)$-flats in $\R^{2d}$. Consider the definable set
    \[ B = \{ (x, v, E), (x,v) \text{ belongs to } E, (x,v,E) \in A \times G \}. \]
   Since the map $\phi : B \to G \times \R, (x, v, E) \mapsto (d_X(x), E)$ is also definable, its fibers admit a finite number of distinct homotopy types by \cite[Section 9]{TameTopology}; in particular there exists a number $M$ such that 
    \[ \forall \, s \, \in (0, t],\, \forall \, E \, \in G, \;\module{\chi(\phi^{-1}(s, E))} \leq M. \]
    Since $\phi^{-1}(s,E) = A_s \cap E$, by Crofton's Formula, there exists a constant $C$ such that for every $0 < s \leq t$, we have
    \begin{align*}
        \mathcal{H}^{d-1}(A_s)  = & \; C \int_{E \in G} \chi(A_s \cap E) \diff E \\
        \leq & \; C \int_{\substack{E \in G \\ E \cap A_s \neq \emptyset}} M \diff E.
    \end{align*} 
    The subset of $G$ consisting of subspaces $E$ with nonempty intersection with at least one of $(A_s)_{s \in (0,t]}$ has finite volume, therefore the quantity $\mathcal{H}^{d-1}(A_s)$ is uniformly bounded in $s$.

    \end{proof}

    \section{WDC sets}

    In this section, we prove that any WDC set is the limit of a sequence of smooth sets with normal cycles of uniformly bounded mass, both with respect to the acyclic convergence and the homotopy distance. We begin by recalling the definitions of d.c. functions and WDC sets.
    
    \begin{definition}[d.c functions and WDC sets]
     A real function $h$ defined on a convex set is called d.c (or \textit{delta-convex}) when it can be written as a difference of two convex functions. It is a \textit{d.c nondegenerate aura} if $h$ is nonnegative and $0$ is a weakly regular value of $h$.
  A compact set $X \subset \R^d$ is called WDC (\textit{weakly delta-convex}) 
     when it is the 0-sublevel set of a nondegenerate d.c aura.
    \end{definition}

We will use the theory of \textit{Monge-Ampère} functions, a class of real functions key to the existence of normal cycles for WDC-sets.

\begin{definition}[Monge-Ampère functions {\cite{MongeAmpere}}]
A locally Lipschitz map $f : \R^d \to \R$ is said to be \emph{Monge-Ampère} when there exists a (necessarily unique) integral $d$-current $[\diff f]$ on $\R^d \times \R^d$ satisfying the following:
\begin{itemize}
    \item $ \partial [\diff f] = 0$ ($[\diff f]$ is a cycle).
    \item $[\diff f]$ is Lagrangian, i.e.,
    \begin{equation*}
    [\diff f] \mres \omega = 0,
    \end{equation*} where $\omega = \sum_{i=1}^d \diff x_i \wedge \diff y_i$ is the canonical symplectic form over $\R^d \times \R^d$.
    \item For any $C^{\infty}$ map $\phi : \R^d \times \R^d \to \R$,  
    \begin{equation*}
    [\diff f](\phi \diff x_1 \wedge \dots \wedge \diff x_d) = \int_{\R^d} \phi(x, \nabla f(x)) \diff x.
    \end{equation*}
    \item  The projection on the first variable of $\R^d \times \R^d$ restricted to the support of $[\diff f]$ is proper.
\end{itemize}
\end{definition}

When $f$ is $C^{1,1}$, its Monge-Ampère current has an explicit representation. Recall that by Rademacher's theorem, the Hessian of $f$ is defined almost everywhere.

\begin{proposition}[Explicit Monge-Ampère currents for a $C^{1,1}$ function {\cite[2.4]{MongeAmpere}}]
\label{prop:MA_exact_mass}
When $f$ is $C^{1,1}$, the Monge-Ampère current is just the graph of its gradient. In particular for any compact rectifiable set $K$ of $\R^d$, we have
    \begin{equation*}
        M([\diff f] \mres K \times \R^d) = \int_{K} \det(I_d + (H_x f)^2)^{1/2} \diff x.
    \end{equation*}    
\end{proposition}

We will repeatedly use the following facts. 

\begin{lemma}[Facts about mollifiers]
\label{lem:convolution}
For a $L$-Lipschitz function $\psi$, denote by $\psi_s$ its convolution with the function $x \mapsto s^{-d}\xi(x/s)$ where $\xi:\R^d \to \R_{\geq 0}$ is a smooth function of integral one with support in the unit ball. Let $s, s'$ be two nonnegative numbers.
    \begin{itemize}
        \item For all $x \in \R^d$, $\module{\psi_s(x) - \psi_{s'}(x)} \leq L\module{s - s'}$.
        \item Assuming that for all compacta $K \subset \R^d$, $\limsup_{s \to 0^+} M( [ \diff \psi_s ] \mres K \times \R^d) < + \infty$, we have in the flat metric
        \begin{equation*}
        \label{eq:MA_continuity}
        \lim_{s \to 0^+} [ \diff \psi_s] = [ \diff \psi].
        \end{equation*}      
        \end{itemize}
\end{lemma}

\begin{proof}
    For the first point, observe that by a change of variable we have 
    \[ \psi_s(x) = \int_{\R^d} \xi(y) \psi(y - sx) \diff y, \]
    so that
    \begin{align*}
        \module{\psi_s(x) - \psi_{s'}(x)} \leq & \int_{\R^d} \xi(y) \module{\psi(y - sx) - \psi(y - s'x)} \diff y   \\
        \leq & \int_{\R^d} \xi(y) \module{s - s'} \norme{y} L \diff y \leq L \module{s - s'}. \\
    \end{align*}
    
    The second point follows from \cite[Proposition 2.7]{MongeAmpere} using the fact that the Lipschitz constants of $\psi_s$ are uniformly bounded and that $\psi_s$ converges to $\psi$ in the $C^0$ topology.
\end{proof}

The following proposition provides a bound on the mass of $[\diff f]$ when $f$ is convex.

\begin{proposition}[Mass bound for Monge-Ampère currents, convex case]
    \label{prop:MA_bound}
        Let $f : \R^d \to \R$ be a convex function (not necessarily $C^{1,1}$), $K \subset \R^d$ be a compact subset of $\R^d$. Then we have
      \begin{equation*}
        M([\diff f] \mres K \times \R^d) \leq \omega_d (\diam(K) + 2\lip(f|_K))^d. 
    \end{equation*}
    where $\omega_d$ is the volume of the unit ball of $\R^d$.
\end{proposition}

\begin{proof}
First assume that $f$ is $C^{1,1}$. From \Cref{prop:MA_exact_mass}, we have 
    \begin{equation*}
        M([\diff f] \mres K \times \R^d) = \int_{K} \det(I_d + (H_x f)^2)^{1/2} \diff x.
    \end{equation*}  
    The Hessian $H_x f$ is semi-definite positive, so that $\det(I_d + H_x f)^2 = \det(I_d + 2 H_x f + (H_x f)^2) \geq \det(I_d + (H_x f)^2)$. Thus we have
        \begin{equation*}
        M([\diff f] \mres K \times \R^d) \leq \int_{K} \det(I_d + H_x f) \diff x. 
        \end{equation*}  
        Now at point $x$ where it is defined, $I_d + H_x f$ is the differential of $\phi : x \mapsto x + \nabla f(x)$. Map $\phi$ is injective as it is the differential of the strictly convex function $x \mapsto \frac{1}{2} \norme{x}^2 + f(x)$, and Lipschitz by assumption. Thus by change of variable we have
        \begin{equation*}
        \int_{K} \det(I_d + H_x f) \diff x = \Vol( \phi(K)).
        \end{equation*}
        The set $\phi(K)$ has diameter bounded by $\diam(K) + 2 \lip(f|_{K})$, whence the desired result in case $f$ is $C^{1,1}$.

Finally, in the general case, $f$ is the limit of the smoothed $f_s$. By the second point of \Cref{lem:convolution}, we have in the flat norm $[\diff f_s] \to [\diff f]$ and the result follows follows from the lower semicontinuity of masses of currents.
\end{proof}

We will make use of the following lemma from \cite[Section 6]{curv_dc}.
    \begin{lemma}[Determinant of difference of matrices]
    \label{lem:matrix}
     Let $A, B$ be square matrices of size $d$. Then, 
     \begin{equation*}
         \det(A - B) = \frac{1}{d!} \sum_{i=0}^d (-1)^i \binom{d}{i} \det( (d-i)A + iB).
     \end{equation*}
     
    \end{lemma}

We use \Cref{lem:matrix} and \Cref{prop:MA_bound} to obtain a bound on the mass of Monge-Ampère currents of d.c functions.

    \begin{lemma}[Mass bound for Monge-Ampère currents, d.c case]
    \label{lem:ma_bound}
    For any function $h : \R^d \to \R$ of the form $h=f-g$ where $f,g$ are convex, we have for every compact subset $K$ of $\R^d$:
    \begin{equation*}
    \label{eq:mineurs}
        M([\diff h] \mres K \times \R^d)) \leq C_d (\diam(K) + 2d\max(\lip(f|_K),\lip(g|_K)))^d,
    \end{equation*}
    with $C_d = \frac{\omega_d}{d!}2^d \binom{2d}{d}$.
    \end{lemma}

    \begin{proof}
    First, assume that $f$ and $g$ are both convex and $C^{1,1}$.    
     For any $I, J \subset \{ 1 , \dots, d \}$ such that $\# I + \# J = d$ and any function $\phi : \R^d \times \R^d \to \R$ the following identity (see for instance \cite[2.8a]{FuSub}):
    \begin{equation}
        \label{eq:expression_ma}
        ([ \mathrm{d} h] \mres K \times \R^d) (\phi \diff x_I \wedge \diff y_J) = \int_{K} \phi(x, \nabla f(x) - \nabla g(x)) \det \left (\frac{\partial^2 (f - g)(x)}{\partial x_i \partial x_j}\right )_{i \notin I, j \in J} \diff x. 
    \end{equation}
    where $\diff x_I, \diff y_J$ are respectively $\bigwedge_{i \in I} \diff x_i$ and $\bigwedge_{j \in J} \diff y_j$. Now take $\psi$ a smooth $d$-differential form with support in $K$ and norm at most 1. Decompose it as
     \begin{equation*}
         \psi = \sum_{\substack{I, J \subset \{ 1, \dots d \}\\ \# I + \# J = d 
    }} \phi_{I,J} \diff x_I \wedge \diff y_J,
     \end{equation*}
     where $\phi_{I,J}$ are smooth functions supported in $K$. By applying $\psi$ to the $n$-vectors drawn to the canonical basis we see that $\norme{\phi_{I,J}}_{\infty} \leq 1$. Plugging this in \Cref{eq:expression_ma} yields
     \begin{equation*}
        [ \mathrm{d} h]( \psi) \leq \int_{\R^d} \sum_{\substack{I, J \subset \{ 1, \dots d \}\\ \# I + \# J = d 
    }} \left | \det \left (\frac{\partial^2 (f - g)(x)}{\partial x_i \partial x_j}\right )_{i \notin I, j \in J} \right | \diff x,
     \end{equation*}
     which can be rewritten as 
          \begin{equation*}
        [\mathrm{d}h]( \psi) \leq \int_{K} \sum_{\substack{I, J \subset \{ 1, \dots d \}\\ \# I = \# J 
    }} \left | \det \left (\frac{\partial^2 (f - g)(x)}{\partial x_i \partial x_j}\right )_{i \in I, j \in J} \right | \diff x.
     \end{equation*}
    Fixing $I,J \subset \{ 1, \dots, d \}$ of same cardinality, we have 
\[
\mathmakebox[1.1\linewidth][c]{
    \displaystyle
        \int_{K} \module{\det \left (\frac{\partial^2 (f - g)(x)}{\partial x_i \partial x_j}\right )_{i \in I, j \in J}} \diff x \leq \frac{1}{d!}\sum_{i=0}^d \binom{d}{i} \int_{K} \module{\det \left (\frac{\partial^2 ( (d-i)f + ig)(x)}{\partial x_i \partial x_j}\right )_{i \in I, j \in J}} \diff x.
}
\]
     by applying \Cref{lem:matrix} to the matrices involved after completing them by the identity to get $d \times d$ matrices.
     
    Now each integral in the right-hand side is the mass of $([\diff \, ((d-i)f + ig)] \mres \diff x_I \wedge \diff y_J)$ on $K \times \R^d$, which is at most the mass of $[\mathrm{d} ((d-i) f + i g)]$ on $K \times \R^d$ by taking $h = (d-i)f + ig$ in \Cref{eq:expression_ma}. Further summing over pairs of subsets $I,J \subset \{ 1, \dots, d \}$ of same cardinality, of which there are $\binom{2d}{d}$, we obtain
    \begin{equation*}
        [\diff h](\psi) \leq \frac{1}{d!} \binom{2d}{d} \sum_{i=0}^d \binom{d}{i} M([\mathrm{d} \left ((d-i)f + ig \right)]).
    \end{equation*}
    Since $(d-i)f + ig$ is $d\max(\lip(f|_K), \lip(g|_K))$-Lipschitz on $K$, the bound of \Cref{prop:MA_bound} yields the desired result.
    
    As for the general case, from \Cref{lem:convolution} we have $[ \diff h_s] \to [ \diff h]$ in the flat norm when $s$ goes to zero, and the desired result follows from the lower semicontinuity of the mass of currents.
    \end{proof}

    The following proposition shows that the mass of the normal cycle of a WDC set is controlled by the mass of the Monge-Ampère current of its aura. 

    \begin{proposition}
    \label{prop:normal_bound}
    Let $h : \R^d \to \R_{\geq 0}$ be a d.c nondegenerate aura with $X \coloneqq h^{-1}(  0 )$ and let $\mu \coloneqq \lim_{t \to 0^+} \inf \{ \Delta (\clarke h(x) ) ,  0 < h(x) < t \}$. Denote $N_X$ the normal cycle of $X$. We have 
    \begin{equation*}
     M(N_X) \leq K_{d, \mu} M( [\diff h] \mres X \times \R^d ).
    \end{equation*}
    where $K_{d, \mu}$ is a constant depending on $d, \mu$.
    \end{proposition}

    \begin{proof}
    Let $\gamma : (x,\xi) \mapsto \norme{\xi}$ and let $\nu: (x, \xi) \mapsto (x, \xi/\norme{\xi})$.
   For every $\eps > 0$, there exists a small enough neighborhood $U$ of $X$ such that 
   \begin{equation*}
       x \in U \text{ and } 0 <\Delta( \clarke h(x)) \leq \mu -\eps \implies x \in \partial X. 
   \end{equation*}
   From the work of Fu \cite{FuSub}, this implies that for any almost every $t$  in $(0, \mu - \eps)$, the normal cycle $N_X$ is the pushforward of the slice $\slice{t}$ by the normalizing map $\nu$:
   \begin{equation*}
    N_X = \nu_{\#}\slice{t}.
   \end{equation*}
   Now the map $\nu$ is $\max(1, t^{-(d-1)})$-Lipschitz on the support of $\slice{t}$, so that writing $\psi(t) = \max(1, t^{-(d-1)})$ we get
   \begin{equation}
   \label{eq:normal_mass_bound}
    M(N_X) \leq \psi(t) M(\slice{t}).   
   \end{equation} 

    Since the map $\gamma$ is 1-Lipschitz and $\psi$ is nonincreasing,
    from \Cref{eq:slice} relating the integrals of slices to the original current, we obtain
   \begin{align*}
       \int_{a}^b \psi(t) M( \slice{t}) \diff t \leq & \psi(a) \int_{a}^b M(\slice{t})\diff t \\
       \leq & \psi(a) M \left( [\diff h] \mres ((U \times \R^d) \cap \gamma^{-1}[a,b]) \right ) \\
       \leq & \psi(a) M([\diff h] \mres (X \times \R^d)).
   \end{align*}
   Finally, we integrate \Cref{eq:normal_mass_bound} for $a \leq t \leq b$ to obtain
   \begin{equation*}
       (b-a) M(N_X) \leq \psi(a) M([\diff h] \mres (X \times \R^d)).
   \end{equation*}
    Letting $a = \mu/2$ and $b$ go to $\mu$ yields the desired result with ${K_{d, \mu} = 2\psi(\mu/2)/\mu}$.
    \end{proof}

Combining the results above gives:

\begin{proposition}[Mass bound for normal cycles of sublevel sets of d.c functions]
\label{prop:sub_dc_aura}
Let $h = f-g: \R^d \to \R_{\geq 0}$ be a d.c function and let $t \in \R$ be such that $\mu = \lim_{\delta \to 0^+} \inf \{\Delta(\clarke \, h(x)), t < h(x) \leq t + \delta \}$ is positive. Then $X_t = h^{-1}(-\infty, t]$ is a WDC set and we have
\begin{equation*}
    M(N_{X_t}) \leq C_d K_{d, \mu}( \diam(X)+ 2d\max(\lip(f|_{X_t}), \lip(g|_{X_t})))^d,
\end{equation*}
where $C_d$ (resp. $K_{d, \mu}$) is the constant in \Cref{lem:ma_bound} (resp. \Cref{prop:normal_bound}). 
    
\end{proposition}

\begin{proof}
    Note that $X_t$ is the zero sublevel set of $h^t = \max(h,t) - t$, and that $\max(f-t,g) - g$ provide a d.c decomposition of $h^t$ sharing the same maximum of Lipschitz constants as the d.c decomposition of $h = f - g$. Moreover the Clarke gradients of $h^t$ and $h$ coincide on the complement of $X_t$, so that $h^t$ is a nondegenerate d.c aura for $X_t$ with 
    \[\lim_{s \to 0^+} \inf \{ \Delta( \clarke h^t(x), 0 < h^t(x) \leq s\} \geq \mu \]
    by upper semicontinuity of Clarke gradients. Putting together the bounds of \Cref{lem:ma_bound} and \Cref{prop:normal_bound}, we have 
    \begin{align*}
        M(N_{X_t}) \leq & \; K_{d,\mu} M([\diff h^t] \mres X_t \times \R^d) \\
        \leq & \; K_{d,\mu} C_d (\diam(X_t) + 2d\max(\lip(f|_{X_t}), \lip(g|_{X_t})))^d.
    \end{align*}
\end{proof}

Let $h$ be a nondegenerate aura. Then there exists $\rho, c > 0$ such that for any $x$ with $0 < h(x) \leq c$,
\[ \Delta( \clarke h(x)) \geq \rho. \]
In particular $h$ satisfies a Kurdyka-Łojasiewicz inequality for some constant function $\phi = \rho,c$ as in \Cref{def:approx_finite_length}. We use this fact to show for small enough smoothing parameters $s$, appropriate sublevel sets of $h_s$ are close to $h^{-1}(0)$ with respect to $\dhom$ and $\dacy$.

\begin{lemma}[Deformation retraction and smoothing]
\label{lem:flot_lisse}
   Let $h: \R^d \to \R_{\geq 0}$ be a $2L$-Lipschitz, nondegenerate aura and let $\rho, c$ be given as in the paragraph above. Then for every $\eps > 0$ and $0 < t \leq c$ there exists a positive $S(t, \eps)$, nondecreasing in $t$, such that for every $0 < s \leq \min(S(t, \eps), (c - t)/4L)$, 
   \begin{itemize}
       \item On $h_s^{-1}(t + 2sL)$, $\norme{\nabla h_s(x)} \geq \rho(1 - \eps)$.
       \item There exists a deformation retraction of $h_s^{-1}[0,t + 2sL]$ onto $X = h^{-1}(0)$ whose trajectories have length at most $(t + 4sL)/(\rho - \eps)$, so that  
       \[ \max( \dhom(X, h_s^{-1}[0,t + 2sL]), \dacy(X, h_s^{-1}[0,t + 2sL])) \leq (t + 4sL)/(\rho - \eps).\]
   \end{itemize}
\end{lemma}

This result is obtained with the help of the following lemma.

\begin{lemma}
\label{lem:double_boule}
Let $\eps > 0$, $x \in h^{-1}(0, c]$, and let $A_x$ be an open neighborhood of $x$ such that for every $y \in A_x$,
\[ \clarke h(y) \subset (\clarke h(x))^{\eps \rho}. \]
Then there exists $\delta(x) > 0$ such that for every $z$ in the open ball $B(x, \delta(x))) \subset h^{-1}(0, + \infty)$ and every $s \leq \delta(x)$, we have 
\[ \nabla h_s(z) \in (\clarke h(x))^{\eps \rho}. \] 
In particular, letting $W(x)$ be the element of least norm in $\clarke h(x)$, we have
\begin{equation}
\label{eq:W_smooth}
\scal{\nabla h_s(z)}{\frac{W(x)}{\norme{W(x)}}} \geq \rho(1 - \eps).
\end{equation}
\end{lemma}

\begin{proof}[Proof of \Cref{lem:double_boule}]
Since $A_x$ is an open neighborhood of $x$ and $h(x) > 0$, there exists $r(x) > 0$ such that the open ball $B(x, r(x))$ is included in both $A_x$ and $h^{-1}(0, + \infty)$. Take $\delta(x) = r(x)/2$.
Since $h_s$ is obtained by convolving $h$ with $x \mapsto s^{-d} \xi(x/s)$, where $\xi$ is a smooth nonnegative map of integral one which has support in the centered ball of radius $1$, we have
for every $0 < s  \leq \delta(x)$ and $z \in B(x, \delta(x))$
\begin{equation*}
\label{eq:double_boule}
\begin{aligned}
\nabla h_s(z) = & \int_{B(z,s)} s^{-d} \xi(u/s) \nabla h( z - u) \diff u \\ 
\in &  \; \overline{\Conv( \cup_{y \in B(z,s) } \clarke h(y)) } \subset \clarke h(x)^{\eps \rho}.\\
\end{aligned}
\end{equation*}
\Cref{eq:W_smooth} follows from \Cref{eq:hyper_support_clarke}.
\end{proof}

\begin{proof}[Proof of \Cref{lem:flot_lisse}]
We refine the construction of \Cref{lem:retractions} to obtain a deformation retraction $C^{*}_{\eps}$ similar to $C_{\eps}$ with the additional property that its restriction to $h_s^{-1}(-\infty, t + 2sL]$ is a deformation of this set onto $X = h^{-1}(0)$ with trajectories of length at most $(t + 4sL)/\rho(1 - \eps)$. 

By upper semicontinuity of the Clarke gradient, for each $x \in h^{-1}(0, c]$, there exists an open set $A_x$ containing $x$ such that for any $y$ in $A_x$, $\clarke h(y) \subset \clarke h(x)^{\eps \rho}$. Let $\delta(x)$ be given by \Cref{lem:double_boule}, and let $B_x = B(x, \delta(x)) \subset h^{-1}(0, + \infty)$. The family $B_x$ forms an open covering of the set $h^{-1}(0, c]$, so that as in the proof of \Cref{lem:retractions} by paracompactness we can extract a countable, locally finite cover $(B_{x_i})_{i \in I}$ of $h^{-1}(0, c]$. Since this covering is locally finite, for any $t \in (0, c]$ the set $I_t$ of $i \in I$ such that $B_{x_i}$ has nonempty intersection with the compact set $h^{-1}[t,c]$ is finite. This implies that $S(t,\eps) = \min_{i \in I_t} \delta(x_i)$ is positive and nondecreasing in $t$.

As in the proof of \Cref{lem:retractions}, letting $\rho_i$ be a smooth partition of unity subordinate to the $B_{x_i}$, the flow $C^*_{\eps}$ of the vector field
\begin{equation*}
V^*(x) = - \sum_{i \in I} \rho_i(x) W(x_i)/\norme{W(x_i)}
\end{equation*}
provides a deformation retraction of $h^{-1}[0, C]$ onto $X$ whose trajectories have length at most $C/\rho(1 - \eps)$ for any $C \in (0, c]$. Moreover for $x \in h^{-1}[t,c]$, \Cref{eq:W_smooth} along with the definition of $V^*(x)$ imply as in the proof of \Cref{lem:retractions} that for any $s \leq S(t, \eps)$, 
\begin{equation}
\label{eq:inegalite_h_s}
\scal{\nabla h_s(x)}{V^*(x)} \leq - \rho(1 - \eps).
\end{equation}

Since $\norme{V^*}_{\infty} \leq 1$, this implies that $\norme{\nabla h_s} \geq \rho(1 - \eps)$ on $h^{-1}[t, c]$. Since $\norme{h_s - h}_{\infty} \leq 2sL$ by the first point of \Cref{lem:convolution}, if $s \leq \min(S(t, \eps), (c-t)/4L)$, $h_s^{-1}(t + 2sL) \subset h^{-1}[t, c]$, so that $t + 2sL$ must be a regular value of $h_s$.

Now for any $0 < s \leq \min(S(t, \eps), (c - t)/4L)$ let $x_0 \in h_s^{-1}[0, t + 2sL]$ and consider the trajectory $u \mapsto C^*_{\eps}(u,x)$, i.e. the trajectory satisfying the differential equation $x'(u) = V^*(x(u))$ until $h(x(u))$ reaches 0, where its stops. Assuming that the maximum of $h_s(x(u))$ is greater than $t + 2sL$ and is attained at a positive $u^*$, the derivative of $h_s(x(\cdot))$ at $u^*$ must be zero, which contradicts \Cref{eq:inegalite_h_s}. Thus $x(t)$ stays in $h_s^{-1}[0, t + 2sL]$ and $C^*_{\eps}$ provides a deformation retraction of $h_s^{-1}[0, t + 2sL]$ onto $X$. Since $h_s^{-1}[0,t + 2sL] \subset h^{-1}[0, t + 4sL]$, the trajectories of this deformation retraction have length at most $(t + 4sL)/\rho(1 -\eps)$.
\end{proof}
In \cite[Section 5]{KinematicWDC}, the authors asked whether any WDC set is the intersection of nested smooth domains with normal cycles of uniformly bounded mass. The following theorem gives a positive answer and shows moreover that WDC sets belong to the class of $C^0$-limits of sequences of smooth sets with uniformly bounded normal cycles.

\begin{theorem}
Let $h = f - g$ be a nondegenerate d.c aura such that $f$ and $g$ are $L$-Lipschitz on a neighborhood of $X = h^{-1}(0)$. Then for any positive decreasing sequence $t_n$ going to zero, there exists a positive sequence $s_n$ such that the sequence of compact domains $X_n \coloneqq h_{s_n}^{-1}[0, t_n + 2s_nL]$ satisfies the following for $n$ sufficiently large:
\begin{itemize}
\item $X_n$ is a smooth compact domain of $\R^d$,
\item $\limsup_n M(N_{X_n}) < + \infty$,
\item $X_{n+1} \subset X_n$,
\item $X = \bigcap_{n} X_n$,
\item $X_n$ converges to $X$ both in the acyclic convergence and with respect to the homotopy distance.
\end{itemize}

\end{theorem}

\begin{proof}

 Since $\norme{h_s - h_{s'}} \leq 2(s - s')L$, observe that when $t - t' \geq 2(s -s')L$, we have 
\begin{equation}
\label{eq:inclusion}
    h_{s'}^{-1}[0, t'] \subset h_s^{-1}[0,t]. 
\end{equation} 
Let $b_n = \min(S(t_1, \eps), \dots, S(t_n, \eps), 1/n)$ for some $0 < \eps < 1$ and $s_{k,n} = b_k + (t_n - t_k)/2L$. Note that $b_n$ is nonincreasing.
Put
\[ s_n = \inf_{k \geq n} s_{k,n}. \]

Observe that for $k \geq n+1$,  $s_{k,n} = s_{k,n+1} + (t_{n} - t_{n+1})/2L$. This yields $s_n \leq s_{k,n+1} + (t_n - t_{n+1})/2L$, and thus $s_n \leq s_{n+1} + (t_n - t_{n+1})/2L$, showing that $X_{n+1} \subset X_n$ by \Cref{eq:inclusion}.
Moreover since $s_{k,n}$ is decreasing in $n$, we have $s_{n+1} = \inf_{k \geq n+1} s_{k,n+1} \leq \inf_{k \geq n+1} s_{k,n}$. Since we also have $s_{n+1, n+1} = b_{n+1} \leq b_n = s_{n,n}$, we get $s_{n+1}\leq s_{n}$.

 Since $s_n \leq 1/n$, $s_n$ must converge to 0, so that for $n$ sufficiently large $s_n \leq (c - t_n)/4L$. By construction $s_n \leq S(t_n, \eps)$. From the first point of \Cref{lem:flot_lisse}, on $h_{s_n}^{-1}(t_n + 2s_nL)$ we have $\norme{\nabla h_{s_n}} \geq \rho(1 - \eps)$, so that $X_n$ is a smooth domain. Moreover the bound of the mass of the normal cycle provided by \Cref{prop:sub_dc_aura} applies. Since the Lipschitz constants of $f_{s_n}$ and $g_{s_n}$ are bounded by that of $f$ and $g$ we have
 \[ \limsup_{n \to \infty} M(N_{X_n}) < + \infty.\]
Furthermore by the second point of \Cref{lem:flot_lisse}, since both $t_n$ and $s_n$ converge to zero, $X_n$ converges both acyclically and with respect to the homotopy distance to $X$. 
Finally if a point $x$ belongs to each $X_n$, then $h_{s_n}(x) \leq t_n + 2s_nL$ and letting $n$ tend to $+\infty$ we obtain $h(x) \leq 0$, so that $\cap_n X_n \subset X$. Conversely $X \subset X_n$ for each $n$ as $\norme{h - h_{s_n}}_{\infty} \leq 2s_nL$.
\end{proof}

\bibliographystyle{alpha} 

\bibliography{ref.bib}

\end{document}